\documentclass[11pt]{article}

\usepackage{amsfonts}
\usepackage{amsmath,amssymb}
\usepackage{amscd}
\usepackage{amsthm}
\usepackage{fancyhdr}
\usepackage{color}
\usepackage{hyperref} 

\theoremstyle{plain}
\newtheorem{lem}{Lemma}[section]
\newtheorem{thm}{Theorem}

\newtheorem{cly}[lem]{Corollary} 
\newtheorem{prop}[lem]{Proposition}

\theoremstyle{definition}
\newtheorem{df}[lem]{Definition} 

\newtheorem{nota}[lem]{Notation} 

\theoremstyle{remark}
\newtheorem{rem}[lem]{Remark}

\newcommand{\ZZ}{\ensuremath{\mathbb Z}}
\newcommand{\CC}{\ensuremath{\mathbb C}}
\newcommand{\RR}{\ensuremath{\mathbb R}}

\newcommand{\cN}{\mathcal{N}}
\newcommand{\cL}{\mathcal{L}}
\newcommand{\cD}{\mathcal{D}}
\newcommand{\cF}{\mathcal{F}}

\newcommand{\cM}{\mathcal{M}}

\newcommand{\cS}{\mathcal{S}}

\newcommand{\cY}{\mathcal{Y}}
\newcommand{\ga}{\mathfrak{a}}
\newcommand{\g}{\mathfrak{g}}

\newcommand{\pd}[1]{\frac{\partial}{\partial #1}} 

\newcommand{\oo}{[\![}
\newcommand{\cc}{]\!]}
\newcommand{\op}{\textbf{\{}}
\newcommand{\cp}{\textbf{\}}}
\newcommand{\li}{\ensuremath{L_{\infty}}}

\begin{document}
\title{Simultaneous deformations and Poisson geometry
}

\author{Ya\"el Fr\'egier\footnote{{  UArtois, LML, F-62\kern 1mm 300, Lens, France.}}\;\footnote{{MIT, 77 Mass. Avenue, Cambridge, MA 02139, USA.}}
\footnote{Universit\"at Z\"urich, Winterthurerstr. 190, CH-8057 Z\"urich, Switzerland. \texttt{yael.fregier@gmail.com}}
\and 
Marco Zambon\footnote{
{ Universidad Aut\'onoma de Madrid }
and
{  ICMAT (CSIC-UAM-UC3M-UCM),}
{ Campus de Cantoblanco,}
{ 28049 Madrid, Spain. } 
\texttt{marco.zambon@uam.es}, \texttt{marco.zambon@icmat.es}} 
\footnote
{Current address: KU Leuven, Department of Mathematics, Celestijnenlaan 200B box 2400, BE-3001 Leuven, Belgium.
{\texttt{marco.zambon@wis.kuleuven.be}}
 }}

\date{}
\maketitle
\begin{abstract}
We consider the problem of deforming simultaneously a \emph{pair} of given structures{. W}e show that {such} deformations are governed by an $\li$-algebra, which we construct explicitly. Our machinery is based on Th. Voronov's derived bracket construction.

In this paper we {consider}  only geometric {applications}, including deformations of coisotropic submanifolds in Poisson manifolds,  of twisted Poisson structures, and of complex structures within generalized complex geometry. These applications can not be, to our knowledge, obtained by other methods such as operad theory. 
\end{abstract}

\tableofcontents

\section*{Introduction}
\addcontentsline{toc}{section}{Introduction}
 
Deformation theory was developed in the 50's by  Kodaira-Kuranishi-Spencer for complex structures \cite{KoSp1}\cite{KoSp2}\cite{KoSp3}\cite{Ku} and by Gerstenhaber  for associative algebras  \cite{Ger}.  Nijenhuis-Richardson then gave an interpretation of deformations in terms of graded Lie algebras (\cite{NijRic1} and \cite{NijRic2}) which was later promoted by Deligne: {\it deformations of a given algebraic or geometric structure $\Delta$     are governed by a differential graded Lie algebra (DGLA) or, more generally, by an $\li$-algebra}.

For example, given a  vector space $V$, Gerstenhaber in \cite{Ger} introduced a graded Lie algebra $(L,[-,-])$ such that an associative algebra structure on $V$  is given by $\Delta\in L_1$ such that $[\Delta,\Delta]=0$.  
A \emph{deformation} of $\Delta$ is an element {$\Delta+  {\tilde{\Delta}}$ such that ${\tilde{\Delta}}\in L_1$ and} \begin{equation}\label{eq:easy}
0=[\Delta+  {\tilde{\Delta}},\Delta+ {\tilde{\Delta}}]= 2[\Delta,  {\tilde{\Delta}}] +[ {\tilde{\Delta}},  {\tilde{\Delta}}] 
=2( {d_\Delta  {\tilde{\Delta}} +\frac{1}{2}[ {\tilde{\Delta}}, {\tilde{\Delta}}]}  ).
\end{equation}
Therefore  the DGLA $(L,d_\Delta, [\cdot,\cdot])$ governs
deformations of the associative algebra $(V,\Delta)$.

 It is usually a hard task to show that the deformations of a given structure are governed by an $\li$-algebra, and even harder to construct  explicitly the $\li$-algebra. When one succeeds in doing so, as a reward one gets the cohomology theory, analogues of Massey products and a natural equivalence relation on the space of deformations. Moreover, quasi-isomorphic $\li$-algebras govern equivalent deformation problems, a result with non-trivial applications to quantization (see \cite{Kont}).\\

In this work we consider \emph{simultaneous} deformations of \emph{two} (interrelated) structures.
A typical example is given by the simultaneous deformations of   $(\Delta,\Phi)$, where $\Delta$ denotes a {pair} of associative algebras and $\Phi$ is an algebra morphism between them. These deformations are characterized by a cubic equation (unlike eq. \eqref{eq:easy} which is quadratic) and are therefore governed by an $\li$-algebra with non trivial $l_3$-term.  

Our main result, Thm. \ref{machine} in \S \ref{heart}, constructs explicitly $\li$-algebras governing such simultaneous deformation problems.\\

\noindent\textbf{Outline of   the paper.}   
\emph{$\li$-algebras}, introduced by Lada and Stasheff \cite{LadaStash}, consist of collections $\{l_i\}_{i\ge 1}$ of ``multi-brackets" satisfying higher Jacobi identities. They can be built out of what we call \emph{V-data} $(L, P, \mathfrak{a}, \Delta)$ via  \emph{derived bracket constructions} due to Th. Voronov   \cite{Vor}   \cite{vor2}, which extend those of Kosmann-Schwarzbach \cite{yvette}  (see  Thm. \ref{voronovderived} and  \ref{voronov}).
Our main contribution {is} to determine $\li$-algebras governing {{\it simultaneous}} deformation problems (Thm. \ref{machine}), by
recognizing that they arise as  in Voronov's Thm. \ref{voronov}.
These results are collected in \S \ref{sec:main}.\\

In the companion paper \cite{FZalg} we find algebraic applications to the study of simultaneous deformations of algebras and morphisms in the following categories: Lie, $L_\infty$,  Lie bi- and associative algebras, and more generally in any category of algebras over Koszul operads.   These results can alternatively be obtained by operadic methods, see for example \cite{FMY} and \cite{MV}, but our techniques have the advantage of not assuming any knowledge of the operadic machinery and of easily delivering  explicit formulae. Recently, using our techniques, Ji studied simultaneous  deformations in the category of Lie algebroids \cite{Ji}.\\

The main novelty {concerning  applications} -- and the focus of this paper -- is in {geometry.} In \S \ref{Poisson}
we determine $L_\infty$-algebras governing {simultaneous} deformations of:
\begin{itemize}
\item coisotropic submanifolds of Poisson manifold{s},   \item Dirac structures in Courant algebroids (with twisted Poisson structures as a special case), \item generalized complex structures in Courant algebroids (with complex structures as a special case).
\end{itemize}  
We also describe explicitly the equivalence relation on the space of twisted Poisson structures. 

None of these examples,  to our knowledge, falls under the scope of the operadic methods, and one should have in mind that in this geometric setting, no tool such as Koszul duality gives for free the graded Lie algebra $L$ we need as part of the V-data.\\

\noindent\textbf{Outlook: deformation quantization of symmetries.} It is known from \cite{BFFLS}  that the quantization of a mechanical system (Poisson manifold) can be understood as a deformation of the algebra of smooth functions ``in the direction" of the Poisson structure, the first order term of the Taylor expansion of this deformation. 

 One can associate to any Poisson structure such a quantization \cite{Kont}:  Poisson structures and their quantizations are Maurer-Cartan elements for suitable $L_\infty$-algebras (Schouten and Gerstenhaber algebras, respectively), so it suffices to build a $L_\infty$-morphism between these two $L_\infty$-algebras  (Formality Theorem). This morphism sends Maurer-Cartan elements to Maurer-Cartan elements, i.e. associates a quantization to any Poisson structure.\\

Our long term goal is to apply this approach to symmetries.
The notion of symmetry of a mechanical system  {$(C^\infty(M),\{-,-\})$} can be understood  as a {Lie algebra map $(\mathfrak{g}, [-,-])\rightarrow (C^\infty(M),\{-,-\})$. This map can be extended,  {in the category of   Poisson algebras, to $(S\mathfrak{g}, \{-,-\})$, the Poisson algebra of polynomial functions on $\mathfrak{g}^*$. Its} graph is a coisotropic submanifold of the Poisson manifold $\mathfrak{g}^*\times M$.}    Therefore, our first step towards this long term goal is to  construct in \S  \ref{cois} an $L_\infty$-algebra governing simultaneous deformations of Poisson tensors and their coisotropic submanifolds.  This
$L_\infty$-algebra plays the role of  {the} Schouten algebra in presence of symmetries. It extends
the $L_\infty$-algebras governing deformations of  coisotropic submanifolds of Poisson manifolds considered by Oh and Park \cite{OP}, and Cattaneo and Felder \cite{CaFeCo2}, since in their setting{s}, the Poisson structure was kept fixed.  \\

\noindent\textbf{Acknowledgements:}  {We thank J. Stasheff for comments, and D. Iacono, M. Manetti, F. Sch\"atz, B. Shoiket, {B. Vallette}, T. Willwacher for useful conversations.} {Further we thank the referee for valuable suggestions that helped improve the paper.}

\noindent M.Z.   thanks Uni.lu for hospitality (Luxembourg, 08/2010, grant FNR/10/AM2c/18). He was partially supported by CMUP (Porto), financed by FCT (programs POCTI, POSI and Ciencia 2007); grants  PTDC/MAT/098770/2008 and 
PTDC/MAT/099880/2008 (Portugal), MICINN RYC-2009-04065, MTM2009-08166-E  and ICMAT Severo Ochoa project SEV-2011-0087 (Spain).
 
\noindent Most of the work of Y.F. on this article was done while assistant of Prof. Dr. Martin Schlichenmaier at Uni.lu  (grant R1F105L15), to whom he would like to address his warmest thanks. He benefited from the support of UAM {through grant MTM2008-02686} (Madrid, 06/2010), and from the MPIM in Bonn (12/2011-01/2012).

 
 
\section{$L_{\infty}$-algebras  via derived brackets and   Maurer-Cartan elements}\label{sec:main}

{The purpose of this section is to establish Thm. \ref{machine}, which produces the $L_\infty$-algebras appearing in the rest of the article. Therefore, we first review some basic material about $L_\infty$-algebras in \S \ref{homolie}, then we recall in \S \ref{vorpaper} Voronov's constructions, which will be   used to establish   Theorem \ref{machine} in \S \ref{heart}.} Our proof is a direct computation, but we also provide a conceptual argument in terms of tangent cohomology, building on  \S \ref{Tang}. {We conclude justifying in \S\ref{sec:conv} why no convergence issues arise in our machinery, and discussing equivalences in \S\ref{sec:sym}.}

\subsection{Background on $L_{\infty}$-algebras}\label{homolie}
 
 We start defining (differential) graded Lie algebras, which are special cases of $\li$-algebras.

\begin{df} A \emph{graded Lie algebra} is a $\ZZ$-graded vector space $L=\bigoplus_{n\in\ZZ}L_n$ equipped with a degree-preserving bilinear bracket $[
\cdot,\cdot] \colon L\otimes L \longrightarrow L$ which satisfies 
\begin{itemize}
\item[1)] graded antisymmetry: $[a,b]=-(-1)^{\vert a\vert\vert b\vert}[b,a]$,
\item[2)] graded Leibniz rule: $[a,[b,c]]=[[a,b],c] + (-1)^{\vert a\vert\vert b\vert}[b,[a,c]].$
\end{itemize}
Here $a, b, c$ are homogeneous elements of $L$ and the degree $\vert x\vert$ of an homogeneous element $x\in L_n$ is by definition $n$. 
\end{df}

 \begin{df} A \emph{differential graded Lie algebra} (DGLA for short) is {a} graded Lie algebra $(L,[\cdot,\cdot])$ equipped with a homological derivation $d \colon L\to L$ of degree 1. In other words:
\begin{itemize}
\item[1)] $\vert da\vert=\vert a\vert +1$ ($d$ of degree 1),  
\item[2)] $d[a,b]=[da,b]+ (-1)^{ \vert a\vert}[a,db]$ (derivation),
\item[3)] $d^2=0$ (homological).
\end{itemize}
\end{df}

In order to formulate the definition of  $L_\infty$-algebra {-- a notion due to Lada and Stasheff \cite{LadaStash}} -- let us give two notations. Given two elements $v_1, v_2$ in a graded vector space $V$, let us define the \emph{Koszul sign}   of the transposition $\tau_{1,2}$ of these two elements by $$\epsilon(\tau_{1,2},v_1,v_2):=(-1)^{\vert v_1\vert \vert v_2\vert}.$$ We then extend multiplicatively this definition to an arbitrary permutation using a decomposition into transpositions. We will often abuse the notation $\epsilon(\sigma,v_1, \dots, v_n)$ by writing $\epsilon(\sigma)$,  and we define $\chi(\sigma):=\epsilon(\sigma)(-1)^{\sigma}$. 

  We will also need \emph{unshuffles}: $\sigma\in S_n$ is called an $(i, n-i)$-unshuffle if it satisfies $\sigma(1)<\dots<\sigma(i)$ and $\sigma(i+1)<\dots<\sigma(n).$ The set of $(i, n-i)$-unshuffles is denoted by $S_{(i, n-i)}.$ Following \cite[Def. 2.1]{LadaMarkl},
  we define
 
\begin{df}\label{li}
An \emph{$L_\infty$-algebra} is a $\ZZ$-graded vector space $V$ equipped with a collection ($k\ge1$) of linear maps $l_k \colon \otimes^kV\longrightarrow V$  of degree $2-k$  satisfying, for every collection of homogeneous elements $v_1, \dots, v_n \in V$:\begin{itemize}
\item[1)] graded antisymmetry: for every $\sigma \in S_n$
$$l_n(v_{\sigma(1)}, \dots, v_{\sigma(n)})=\chi(\sigma)l_n(v_1,\dots,v_n),$$
\item[2)] relations: for all $n\ge 1$
$$\sum_{\substack{i+j=n+1\\i,j\ge 1}}(-1)^{i(j-1)}\sum_{\sigma\in S_{(i,n-i)}}\chi(\sigma)l_j(l_i(v_{\sigma(1)}, \dots,v_{\sigma(i)}),v_{\sigma(i+1)}, \dots, v_{\sigma(n)} )=0.$$
\end{itemize}
In a \emph{curved $L_\infty$-algebra} one additionally allows for an element $l_0\in V_2$, one allows $i$ and $j$ to be zero in the relations   2), and one adds the relation corresponding to  $n=0$. 
\end{df}

Notice that when all $l_k$ vanish except for $k=2$, we obtain graded Lie algebras.

In Def. \ref{li}  the multibrackets are graded antisymmetric and  $l_k$ has degree $2-k$, whereas in the next definition they are graded symmetric and all of degree $1$.

\begin{df}\label{li1}
An \emph{$L_\infty[1]$-algebra} is a graded vector space $W$ equipped with a collection  ($k\ge1$) of linear maps $m_k \colon \otimes^kW\longrightarrow W$ of degree $1$ satisfying, for every collection of homogeneous elements $v_1, \dots, v_n \in W$:\begin{itemize}
\item[1)] graded  symmetry: for every $\sigma \in S_n$
$$m_n(v_{\sigma(1)}, \dots, v_{\sigma(n)})=\epsilon(\sigma)m_n(v_1,\dots,v_n),$$
\item[2)] relations:  for all $n\ge 1$
$$
\sum_{\substack{i+j=n+1\\i,j\ge 1}}
 \sum_{\sigma\in S_{(i,n-i)}}\epsilon (\sigma)m_j(m_i(v_{\sigma(1)}, \dots,v_{\sigma(i)}),v_{\sigma(i+1)}, \dots, v_{\sigma(n)} )=0.$$
\end{itemize}
In a \emph{curved $L_\infty[1]$-algebra} one additionally allows for an element $m_0\in W_1$ (which can be understood as a   bracket with zero arguments),  one allows $i$ and $j$ to be zero in the relations   2), and one adds the relation corresponding to  $n=0$.   
\end{df}

\begin{rem} 
\label{desuspend}
There is a bijection between \li-algebra structures on a graded vector space $V$ and $\li[1]$-algebra structures on $V[1]$, the graded vector space defined by $(V[1])_i:=V_{i+1}$ \cite[Rem. 2.1]{Vor}. The multibrackets are related by applying the d\'ecalage isomorphisms

\begin{equation}\label{deca}
 (\otimes^n V)[n] \cong \otimes^n(V[1]),\;\; v_1\dots v_n \mapsto v_1\dots v_n\cdot (-1)^{(n-1)|v_{1}|+\dots+2|v_{n-2}|+|v_{n-1}|},
 \end{equation}
 where $|v_i|$ denotes the degree of $v_i\in V$. The bijection extends to the  curved case.
 \end{rem}
 
 From now on, for any $v\in V$, we denote by $v[1]$ the corresponding element in $V[1]$ (which has  degree $|v|-1$). Also, we denote the multibrackets in $\li[1]$-algebras by $\{\cdots\}$, we denote by $d:=m_1$ the unary bracket, and in the curved case we denote $\{\emptyset\}:=m_0$ (the   bracket with zero arguments).

\begin{df} Given an $L_\infty[1]$-algebra $W$, a \emph{Maurer-Cartan element} is a degree $0$ element $\alpha$ satisfying the 
 Maurer-Cartan equation
   \begin{equation}\label{MaCa}
 \sum_{n=1}^{\infty} \frac{1}{n!}\{\underbrace{\alpha,\dots,\alpha}_{n \text{ times}}\}=0.
\end{equation}
One denotes by $MC(W)$ the set of its Maurer-Cartan elements. 

If $W$ is a {curved} $L_\infty[1]$-algebra, one defines Maurer-Cartan elements by adding $m_0\in W_1$ to the left hand side of eq. \eqref{MaCa} (i.e. by letting the sum in \eqref{MaCa} start at $n=0$).
\end{df}
 
There is an issue with the above definition: the l.h.s. of eq. \eqref{MaCa} is generally an infinite sum.   In this paper we solve this issue by considering \emph{filtered} $L_\infty[1]$-algebras (see Def. \ref{def:filteredli}), for which the above infinite sum   automatically converges.



\subsection{Th. Voronov's constructions of $L_\infty$-algebras as derived brackets}\label{vorpaper}
In this subsection we introduce V-data and recall how Voronov associates $L_\infty[1]$-algebras to a V-data.
 
\begin{df}\label{vdata}
 A \emph{V-data} consists of a quadruple $(L,\ga, P,\Delta)$ where 
 \begin{itemize}
\item  $L$ is a graded Lie algebra (we denote its bracket by $[\cdot,\cdot]$), 
\item $\ga$ an abelian  Lie subalgebra,  
\item $P \colon L \to \ga$ a projection whose kernel is a Lie subalgebra of $L$, 
 \item $\Delta \in Ker(P)_1$ an element such that $[\Delta,\Delta]=0$.
\end{itemize}
When  $\Delta$ is an arbitrary element of $L_1$ instead of $Ker(P)_1$, we refer to  $(L,\ga, P,\Delta)$
as a \emph{curved V-data}. 
\end{df}

\begin{thm}[{\cite[Thm. 1, Cor. 1]{Vor}}]\label{voronovderived}
 Let $(L,\ga, P, \Delta)$ be a curved V-data.
Then $\mathfrak{a}$ is a curved $L_\infty[1]$-algebra for the multibrackets   $\{\emptyset\}:=P\Delta$ and 
($n\ge 1$)
\begin{align}\label{eq:adp}
\{ a_1,\dots,a_n\}&=P[\dots[[\Delta,a_1],a_2],\dots,a_n].
\end{align} 
We obtain a  $L_\infty[1]$-algebra exactly when $\Delta \in Ker(P)$ .
\end{thm}

When $\Delta \in Ker(P)$  there is actually a larger  $L_\infty[1]$-algebra, which contains $\ga$ as in Thm. \ref{voronovderived} as a $L_\infty[1]$-subalgebra.

\begin{thm}[{\cite[Thm. 2]{vor2}}]\label{voronov}
Let $V:=(L,\ga, P, \Delta)$ be a V-data, and denote $D:=[\Delta,\cdot] \colon L \to L$.
Then the space $L[1]\oplus \mathfrak{a}$ is a $L_\infty[1]$-algebra for the differential
\begin{equation}\label{diffV2}
d( x[1], a):= (- (Dx)[1], P(x+Da)),
\end{equation}
the binary bracket
\begin{equation}\{x[1],y[1]\}=[x,y][1](-1)^{|x|}\in L[1], \label{crochet}
\end{equation}
and for $n\ge 1$:
\begin{align}\label{vorder}
\{ x[1],a_1,\dots,a_n\}&=P[\dots[x,a_1],\dots,a_n]\in \ga,\\
\label{vorderlong}\{a_1,\dots,a_n\}&=P[\dots[Da_1,a_2],\dots,a_n]\in \ga.
\end{align}
Here $x,y\in L$ and $a_1,\dots,a_n \in \mathfrak{a}$. Up to permutation of the entries, all the remaining multibrackets vanish.
\end{thm} 

\begin{nota}
We will denote by $$\ga_{\Delta}^P$$ and by $$(L[1]\oplus \mathfrak{a})_{\Delta}^P, \ or \ sometimes \ \g(V)$$ the $L_{\infty}[1]$-algebras produced by Thm. \ref{voronovderived} and \ref{voronov}.  \end{nota}

Given a curved $V$-data, assume that $\Phi\in \ga_0$ is such that  $e^{[\cdot,\Phi]}$ is  well-defined (see Prop. \ref{prop:fil} for a sufficient condition), giving an automorphisms of $(L,[\cdot,\cdot])$. We will consider
\begin{equation}\label{pphi}
P_{\Phi} :=P \circ e^{[\cdot,\Phi]}\colon L \to \ga.
\end{equation}
Notice that  $P_\Phi$ is a projection since $e^{[\cdot,\Phi]}|_{\ga}=Id_{\ga}$ by the abelianity of $\ga$.

\begin{rem}\label{MCexp}  Let $(L,\ga, P, \Delta)$ be a curved V-data and
  $\Phi\in \ga_0$ as above. Then $\Phi$ is a Maurer-Cartan element of $\ga_{\Delta}^P$ if{f} 
\begin{equation}
P_\Phi \Delta=0,
\end{equation}
or equivalently $\Delta\in ker(P_{\Phi})$. This follows immediately from eq. \eqref{eq:adp} and will be used repeatedly in the proof of Thm. \ref{machine}. 
\end{rem}

\begin{rem}\label{kerPa}
Let $L'$ be a graded Lie subalgebra of $L$ preserved by $D$ (for example $L'=Ker(P)$). Then $L'[1]\oplus\ga$ is stable under the multibrackets of Thm. \ref{voronov}. We denote by $(L'[1]\oplus \mathfrak{a})_{\Delta}^P$ the induced $\li[1]$-structure.  
\end{rem}

\begin{rem}\label{DKK}
Voronov's \cite[Thm. 2]{vor2} is actually formulated for any  degree 1
derivation $D$ of $L$ preserving $Ker(P)$ and satisfying $D\circ D=0$. We restrict ourselves to inner derivations for the sake of simplicity, and since all the derivations that appear in our examples are of this kind.

A ``semidirect product''   $L_\infty[1]$-algebra similar to the one in Thm. \ref{voronov} appeared in \cite{BFLS}   \cite{MapCone}.
\end{rem}



\subsection{{The tangent complex within Voronov's theory.}}\label{Tang}

In this subsection we study how Voronov's $\li[1]$-algebras behave under twisting. We will use this in \S\ref{heart} to provide an alternative argument for Thm. \ref{machine}.

It is well known  \cite[Prop.4.4]{Getzler} that one can twist an $L_\infty[1]$-algebra $\g$ by one of its Maurer-Cartan elements $\alpha$. One obtains a new $L_\infty[1]$-algebra $\g_\alpha$, sometimes called the \emph{tangent complex at $\alpha$}. Its $n$-th multibracket is  
\begin{equation}\label{eq:twb}
\{\dots\}^\alpha_n =\{\dots\}_n+\{\alpha,\dots\}_{n+1} +\frac{1}{2!}\{\alpha,\alpha,\dots\}_{n+2}+ \dots
\end{equation}
where $\{\dots\}_j$ denotes the $j$-th multibracket of $\g$.
 
A property of the tangent complex $\g_\alpha$ is that its Maurer-Cartan elements are in one to one correspondence with the deformations of $\alpha$, i.e.
\begin{equation}\label{eq:MCtw}
 \alpha+ \tilde{\alpha} \in MC(\g)
\;\;\;\Leftrightarrow \;\;\;
\tilde{\alpha} \in MC(\g_\alpha)
\end{equation}
(\cite[Prop. 12.2.33]{LodVal}   or direct computation). 
We express the notion of tangent complex in the setting of Voronov's theory (recall that the notation $\g(V)$ was defined in \S\ref{vorpaper}):

   \begin{lem}\label{truc}
Let $V:=(L,\ga, P,\Delta)$ be a  filtered V-data and let $\alpha :=(\Delta'[1],\Phi')$
be a Maurer-Cartan element of $\g(V)$. Then 
$$\g(V)_\alpha=\g(V_\alpha),$$
with $V_\alpha:=(L,\ga, P_{\Phi'},\Delta+\Delta').$
\end{lem}

 This lemma is a generalization of the remark by Domenico Fiorenza that  $((L[1]\oplus \mathfrak{a})_{0}^P)_{(\Delta[1],0)}=(L[1]\oplus \mathfrak{a})_{\Delta}^P$. 
 We do not need to prove that $V_\alpha$ is a V-data since, as a twist of a $L_\infty[1]$-algebra, $\g(V_\alpha)$ is automatically a $L_\infty[1]$ algebra.

\begin{proof} Let $n>2$ and
  $\alpha:=(\Delta'[1],\Phi')\in L[1]\oplus \ga$. The $k$-th summand  ($k\ge 0$) of the r.h.s. of  equation \eqref{eq:twb}, applied to elements $x_i[1]+a_i$, can be rewritten as
 \begin{eqnarray*}
  \frac{1}{k!}\{\underset{k}{\underbrace{\alpha,\dots\alpha}},x_1[1]+a_1,\dots,x_n[1]+a_n\}_{n+k}  & = & A_k+B_k+C_k
 \end{eqnarray*}
 with
\begin{align}
A_k &=\frac{1}{(k-1)!} P[\dots[\dots[\Delta',\underset{k-1}{\underbrace{\Phi'], \dots, \Phi'}}], a_1,\dots,a_n],\label{clubk}\\
B_k &=\frac{1}{k!} \sum_{i=1}^n P[\dots[\dots[x_i,\underset{k}{\underbrace{\Phi'], \dots, \Phi'}}], a_1,\dots,\hat{a_i}\dots,a_n],\label{heartk}\\
C_k &=\frac{1}{k!} P[\dots[\dots[\Delta,\underset{k}{\underbrace{\Phi'], \dots, \Phi'}}], a_1,\dots,a_n]\label{diamondk},
\end{align}
defining $A_0:=0$.
Notice that \eqref{clubk} and \eqref{heartk} come from \eqref{vorder} (encoding  the $L[1]$-components of $\alpha$ and $x_i[1]+a_i$ respectively), and \eqref{diamondk} comes from \eqref{vorderlong}.

On the other hand the  brackets of $(L[1]\oplus \mathfrak{a})_{\Delta+\Delta'}^{P_{\Phi'}}$ for $n>2$ read
\begin{eqnarray*}
  \{x_1[1]+a_1,\dots,x_n[1]+a_n\}_{n}  & = & A+B+C
 \end{eqnarray*}
where
\begin{align*}
A &= P_{\Phi'}[\dots[\Delta', a_1],\dots,a_n],\\
B &=\sum_{i=1}^n P_{\Phi'}[\dots[x_i, a_1],\dots,\hat{a_i}\dots,a_n],\\
C &=P_{\Phi'}[\dots[\Delta,a_1],\dots,a_n].
\end{align*}
Since $e^{[\cdot,\Phi']}$ is a morphism of graded Lie algebras and   $e^{[\cdot,\Phi']}|_{\ga}=Id_{\ga}$ we have
$$e^{[\cdot,\Phi']}[\dots[x, a_1],\dots,a_n]=[\dots[e^{[\cdot,\Phi']}x, a_1],\dots,a_n]$$ for all $x\in L$.
Expanding $e^{[\cdot,\Phi']}$ as a series gives
$$A=\sum_k A_k,\;\;\;\;B=\sum_k B_k,\;\;\;\;C=\sum_k C_k,$$
therefore showing that the $n$-th multibrackets agree for $n>2$ .
Similar computations give the cases $n=1, 2$.  
\end{proof}



\subsection{The main tool}\label{heart}

Given a  V-data $(L,\ga, P,\Delta)$, we fix a Maurer-Cartan $\Phi$ of $\ga_{\Delta}^P$ and study the deformations of $\Delta$ and $\Phi$.  

In what follows, the  assumption  \emph{filtered} is there to ensure the convergences of the infinite sums appearing, and can be neglected on a first reading. We will address convergence issues in \S \ref{sec:conv}.

\begin{lem}\label{twist Vdata}
Let $(L,\ga, P,\Delta)$ be a filtered
V-data and let $\Phi\in MC(\ga_{\Delta}^P)$. Then 
$(L,\ga, P_\Phi,\Delta)$ is also a   V-data.    
   \end{lem}

  \begin{proof}
   $P_{\Phi}$ is well-defined  in Prop. \ref{prop:fil} in \S\ref{sec:conv}. $Ker (P_{\Phi})=e^{[\cdot,-\Phi]}(Ker(P)) $ is a Lie subalgebra of $L$ since  $e^{[\cdot,-\Phi]}$ is a Lie algebra automorphism of $L$ and $ker(P)$ is a Lie subalgebra. Further $\Delta \in ker(P_{\Phi})$ by Remark \ref{MCexp}. Hence $(L,\ga, P_\Phi,\Delta)$ is a V-data.
 \end{proof}
   
The following is the main tool used in the rest of the paper. It says that the deformations of $\Delta$ and $\Phi$ are governed by 
$(L[1]\oplus \ga)_{\Delta}^{P_{\Phi}}$.
In the applications, $\Phi$  will be the object of interest, as it will correspond to morphisms, subalgebras, etc.

   \begin{thm}\label{machine}
Let $(L,\ga, P,\Delta)$ be a  filtered V-data and let $\Phi\in MC(\ga_{\Delta}^P)$. Then for all   $\tilde{\Delta}\in L_1$   and  $\tilde{\Phi}\in \ga_0$:
\begin{align}
\begin{cases}
 [\Delta+\tilde{\Delta},\Delta+\tilde{\Delta}]=0  \\
\Phi+ \tilde{\Phi} \in MC(\ga_{\Delta+\tilde{\Delta}}^{P})\end{cases}
\Leftrightarrow
(\tilde{\Delta}[1],\tilde{\Phi}) \in MC((L[1]\oplus \ga)_{\Delta}^{P_{\Phi}}).
\end{align} 
In this case, $\ga^P_{\Delta+\tilde{\Delta}}$  is a curved $\li[1]$-algebra. It is a  $\li[1]$-algebra exactly when $\tilde{\Delta}\in Ker(P)$.
\end{thm}

\begin{proof}{By Lemma \ref{twist Vdata}  we can apply Thm. \ref{voronov} to obtain the $\li[1]$-algebra $(L[1]\oplus \ga)^{P_{\Phi}}_{\Delta}$, whose multibrackets we
 denote by $\{\dots\}$.}
  We compute each summand appearing in the l.h.s of the Maurer-Cartan equation for $(\tilde{\Delta}[1],\tilde{\Phi})$ in $(L[1]\oplus \ga)^{P_{\Phi}}_{\Delta}$, which reads 
 \begin{equation}\label{MCtilde}
\sum_{n=1}^{\infty}\frac{1}{n!}\{(\tilde{\Delta}[1],\tilde{\Phi}),\dots,(\tilde{\Delta}[1],\tilde{\Phi})\}. 
\end{equation}
 We have 
 \begin{align*}
 \{(\tilde{\Delta}[1],\tilde{\Phi})\}&=(-[\Delta,\tilde{\Delta}][1], \;\;P_{\Phi}\tilde{\Delta}\;\;\;\;\;\;\;\;\;\;\;\;\;\;\;\;\;\;\;\;\;\;\;\;\;\;\;+P_{\Phi}[\Delta, \tilde{\Phi}]\;\;\;\;\;\;\;\;\;\;\;\;\;\;\;\;\;\;\;\;),\\
\{(\tilde{\Delta}[1],\tilde{\Phi}),(\tilde{\Delta}[1],\tilde{\Phi})\}&=(-[\tilde{\Delta},\tilde{\Delta}][1], \;\;2\cdot P_{\Phi}[\tilde{\Delta}, \tilde{\Phi}] \;\;\;\;\;\;\;\;\;\;\;\;\;\;\;\;+P_{\Phi}[[\Delta, \tilde{\Phi}],\tilde{\Phi}]\;\;\;\;\;\;\;\;\;\;\;\;\;\;),\\
\{\underbrace{(\tilde{\Delta}[1],\tilde{\Phi}),\dots,(\tilde{\Delta}[1],\tilde{\Phi})}_{n \text{ times}}\}&=
 (\;\;\;\;\;0\;\;\;\;\;\;\;\;\;\;,\;\; n \cdot P_{\Phi}[[[\tilde{\Delta}, \underbrace{\tilde{\Phi}],\dots],\tilde{\Phi}}_{n-1 \text{ times}}] +P_{\Phi}[[[[\Delta, \underbrace{\tilde{\Phi}],\tilde{\Phi}],\dots],\tilde{\Phi}}_{n \text{ times}}]).
 \end{align*}
The last line refers to the $n$-th term for $n \ge 3$, and holds
 since   the higher brackets with two or more entries in $L[1]\oplus \{0\}$ vanish.
 
Hence the $L[1]$-component of \eqref{MCtilde} is just  
 $-\frac{1}{2} [\Delta+\tilde{\Delta},\Delta+\tilde{\Delta}][1]$. 
The $\ga$-component of \eqref{MCtilde} is
\begin{align*}
&P_{\Phi}\left( e^{[\cdot, \tilde{\Phi}]}\tilde{\Delta}+ (e^{[\cdot, \tilde{\Phi}]}-1){\Delta}\right)\\
=&P_{\Phi}e^{[\cdot, \tilde{\Phi}]}(\Delta+\tilde{\Delta})\\
=&P e^{[\cdot, \Phi+\tilde{\Phi}]}(\Delta+\tilde{\Delta}),
\end{align*}
which by Remark \ref{MCexp} is the l.h.s. of the Maurer-Cartan equation in $\ga_{\Delta+\tilde{\Delta}}^{P}$ 
 for $ \Phi+\tilde{\Phi}$. Here in the first equation we used   Remark \ref{MCexp}.

{The last two statements follow from Thm. \ref{voronovderived}.}
\end{proof}

We end this subsection presenting
an alternative, more conceptual proof of Thm. \ref{machine}. It is given by:
 \begin{eqnarray*}
 (\tilde{\Delta}[1],\tilde{\Phi}) \in MC((L[1]\oplus \ga)_{\Delta}^{P_{\Phi}})      & \Leftrightarrow &    ((\Delta+\tilde{\Delta})[1],\Phi+\tilde{\Phi}) \in MC((L[1]\oplus \ga)_{0}^{P})
\\
&  \Leftrightarrow & \begin{cases}
 [\Delta+\tilde{\Delta},\Delta+\tilde{\Delta}]=0  \\
\Phi+ \tilde{\Phi} \in MC(\ga_{\Delta+\tilde{\Delta}}^{P}).\end{cases}
\end{eqnarray*}
 The first equivalence is the conjunction of Lemma \ref{truc} (applied to $V=(L,\ga,P,0)$ and $\alpha=(\Delta[1],\Phi)$) and
 of property \eqref{eq:MCtw}. The second equivalence comes from the fact that the only non-vanishing brackets of $(L[1]\oplus \ga)_{0}^{P}$ are given by $d(x[1])= Px$ for $x\in L$, by  \eqref{crochet}  and \eqref{vorder}.



\subsection{Convergence issues}\label{sec:conv}

The left hand side of the Maurer-Cartan equation \eqref{MaCa} is generally an infinite sum.
In this subsection we review Getzler's notion of filtered $\li$-algebra \cite{GetzlerGoe}, which guarantees that the above infinite sum converges. We show that simple assumptions on   V-data ensure that the Maurer-Cartan equations of the (curved) $\li[1]$-algebras we construct in  Thm. \ref{machine} (and {Lemma \ref{truc}}) do converge.

\begin{df}
Let $V$ be a graded vector space.  A \emph{complete filtration}  is   a descending filtration by graded subspaces
$$V=\cF^{-1}V\supset\cF^0V\supset \cF^1V\supset \dots$$
such that the canonical projection
$V \to \underset{\leftarrow}{\lim} V/\cF^nV$ {is an isomorphism}.  {Here $$\underset{\leftarrow}{\lim} V/\cF^nV {:=}\{\overset{\rightarrow}{x}\in {\Pi_{n\geq -1}}V/\cF^nV\;:\; P_{i,j}({x_j)=x_i} \text{ when }  i<j\},$$ where $P_{i,j} \colon  V/\cF^jV\longrightarrow V/\cF^iV$ is the canonical {projection} induced by the inclusion ${\cF^jV\subset \cF^i V}$.}
\end{df}
\begin{rem}\label{rem:dirprod}
If $V$ can be written as a direct product of subspaces $V=\prod_{k\ge -1} V^k$, then $\{\cF^nV\}_{n\ge -1}$ is a complete filtration of $V$, where $\cF^nV:=\prod_{k\ge n}V^k$.
\end{rem}

\begin{df}\label{def:filteredli}
Let $W$ be a curved $\li[1]$-algebra. We say that $W$ is \emph{filtered}\footnote{Our definition differs from Getzler's, which requires that $W=\cF^{0}W$ and that the multibrackets have filtration degree zero except for   the zero-th bracket which  has filtration degree one.}  if there exists a complete filtration on the vector space $W$
such that
all multibrackets $\{\dots\}$ have filtration degree $-1$.
\end{df}
Notice that for an element $\Phi\in W$ of   filtration degree $1$, we have  $\{\Phi,\dots,\Phi\}_n\in \cF^{n-1}W$ for all $n$, so the infinite sum
\begin{equation}\label{eq:mclhs}
\sum_{n=0}^{\infty} \frac{1}{n!}\{ {\Phi,\dots,\Phi}\}_n
\end{equation}
 converges in $W$ by the completeness of the filtration.   Indeed, setting
 $w_i:=\sum_{n=0}^{i} \frac{1}{n!}\{ {\Phi,\dots,\Phi}\}_n$ mod $\cF^i W$ for all $i$ defines 
 an element $\overset{\rightarrow}{w}\in \Pi_{n\geq -1}W/\cF^nW$ which turns out to belong to $\underset{\leftarrow}{\lim} W/\cF^nW\cong W$.

 We define \emph{Maurer-Cartan elements} to be  $\Phi\in W_0\cap \cF^1W$ for which the infinite sum \eqref{eq:mclhs} vanishes, and we write $MC(W)$ for the set of Maurer-Cartan elements.

\begin{df}\label{triple}
{Let $(L,\ga, P,\Delta)$ be a curved V-data (Def. \ref{vdata}). 
 We say that this {curved V-data} is }
\emph{filtered} if  
 there exists a complete filtration on the graded vector space $L$ 
such that
\begin{itemize}
\item[a)] The Lie bracket has filtration degree zero, i.e. $[\cF^iL,\cF^jL]\subset \cF^{i+j}L$ for all $i,j \ge -1$,
\item[b)] $\ga_0 \subset \cF^1L$,
\item[c)] the projection $P$ has filtration degree zero, i.e. $P(\cF^iL)\subset 
\cF^iL$ for all $i\ge -1$.
\end{itemize}
\end{df}

\begin{prop}\label{prop:fil}
 Let $(L,\ga, P,\Delta)$ be a  
{filtered,}  curved {V-data}.
  Then for every $\Phi\in MC(\ga_{\Delta}^P)\subset \ga_0$:
\begin{itemize} 
\item[1)] the projection $P_{\Phi} :=P \circ e^{[\cdot,\Phi]}\colon L \to \ga$ is  well-defined and has filtration degree zero.
\item[2)]  the curved $\li[1]$-algebra   $\ga_{\Delta}^{P_{\Phi}}$ given by Thm. \ref{voronovderived} is filtered by $\cF^n\ga:=\cF^nL\cap \ga$. Further, the sum \eqref{eq:mclhs}  converges for any degree zero element $a$ of $\ga$.

\item[3)] if $\Delta \in ker(P)$: the   $\li[1]$-algebra $(L[1]\oplus \mathfrak{a})_{\Delta}^{P_{\Phi}}$ given by Thm. \ref{voronov} is filtered by $\cF^n(L[1]\oplus \mathfrak{a}):=(\cF^nL)[1]\oplus \cF^n\ga$. Further, the sum \eqref{eq:mclhs}  converges for any degree zero element  element $(x[1],a)$ of $L[1]\oplus \ga$.
\end{itemize}
\end{prop}
\begin{proof}
1) For every $x\in L$, say $x\in \cF^{i}L$, by Def. \ref{triple} a)b) we have
$$[[\dots[x,\underbrace{\Phi],\dots],\Phi}_{n \text{ times}}]\in \cF^{i+n}L.$$ Hence  the completeness of the filtration on $L$ implies that 
$e^{[\cdot,\Phi]}$ is a well-defined endomorphism of $L$.
The above also shows that $e^{[\cdot,\Phi]}$ has filtration degree zero, and since $P$ does by Def. \ref{triple} c), we conclude that 
the projection $P_{\Phi}$  has filtration degree zero.

2) We first check that $ \{\cF^n\ga\}_{n\ge -1}$ is a complete filtration of the vector space $\ga$.

{The map $\ga \to \underset{\leftarrow}{\lim} \ga/\cF^n\ga$ is
surjective}. Indeed, take an element of $\underset{\leftarrow}{\lim}
\ga/\cF^n\ga$, and consider its image under the canonical embedding
$\underset{\leftarrow}{\lim} \ga/\cF^n\ga \hookrightarrow
\underset{\leftarrow}{\lim} W/\cF^nW$. It is a sequence of elements
$\{ a_i \text{ mod }\cF^i W\}_{i \ge -1}$ where $a_i\in \ga$. The
surjectivity of $W \to \underset{\leftarrow}{\lim} W/\cF^nW$ implies
that there is an element $w\in W$ such that $a_i \text{ mod } \cF^i
W=w \text{ mod } \cF^i W$ for all $i$, which implies $w\in \cF^i
W+\ga$ for all $i$ and hence
$w \in \cap_i (\cF^i W+\ga)$. Since $ \cap_i (\cF^i W)=\{0\}$ (by the
injectivity of $W \to \underset{\leftarrow}{\lim} W/\cF^nW$), this
means that $w\in \ga$.

{The map $\ga \to \underset{\leftarrow}{\lim} \ga/\cF^n\ga$ is injective}. Indeed, an element $a \in \ga$ is sent to 0 if and only if $a\in \cap_i (\cF^i \ga)$. But $\cap_i (\cF^i \ga)\subset \cap_i (\cF^i W)$, which is $\{0\}$ {as seen above.}

The   multibrackets of $\ga_{\Delta}^{P_{\Phi}}$ is given by
$P_{\Phi}[\dots[[\Delta,\bullet],\bullet],\dots,\bullet]$
(see Thm. \ref{voronovderived}). Using 1) and Def. \ref{triple} a), we see that this multibracket has filtration degree $-1$.

For the last statement, notice that $\ga_0 \subset \cF^1\ga$
by Def. \ref{triple} b).

3) $\{(\cF^nL)[1]\oplus \cF^n\ga\}_{n\ge -1}$ is a complete filtration  of the vector space $L[1]\oplus \mathfrak{a}$ because   the two summands are complete filtrations of $L[1]$ and $\ga$ respectively (by assumption and by 2) respectively).
The multibrackets of    $(L[1]\oplus \mathfrak{a})_{\Delta}^{P_{\Phi}}$
are given in Thm.  \ref{voronov}, and all have filtration degree $-1$ by 1) and 
Def. \ref{triple} a). 

For the last statement, notice that the non-vanishing multibrackets of $(L[1]\oplus \mathfrak{a})_{\Delta}^{P_{\Phi}}$ accept at most two entries from $L[1]$, and use again  $\ga_0 \subset \cF^1\ga$.
 \end{proof}

A version of Prop. \ref{prop:fil} in which the curved V-data is not assumed to be filtered, and working in the formal setting, is given in \cite{FZalg}.

\subsection{Equivalences of  Maurer-Cartan elements}\label{sec:sym}
 
Let $W$ be an $\li[1]$-algebra. On  $MC(W)$, the set of Maurer-Cartan elements, there is a canonical  
involutive (singular) distribution $\cD$   which induces an equivalence relation on $MC(W)$ known as \emph{gauge equivalence}. 
More precisely, each $z\in W_{-1}$ defines a vector field $\cY^z$ on $W_0$, whose value at $m\in W_0$ is\footnote{The infinite sum   \eqref{gaugeaction} is guaranteed to converge if $W$ is  filtered and $W_{-1}\subset \cF^{1}W$, see \S \ref{sec:conv}. For  the example we consider in \S\ref{sec:tpois}, this sum is actually finite.}
  
\begin{equation}\label{gaugeaction}
\mathcal{Y}^z|_m:=dz+\{z,m\}+\frac{1}{2!}\{z,m,m\}+\frac{1}{3!}\{z,m,m,m\}+\dots.
\end{equation}
This vector field is tangent to $MC(W)$. The distribution at the point $m\in MC(W)$ is defined as
$\cD|_m=\{\mathcal{Y}^z|_m: z\in W_{-1}\}$.

\begin{rem} We give a justification of the above statements, see also 
\cite[\S 3.4.2]{KontSoib} \cite[\S 2.5]{MerkLI}\cite[\S 2.2]{FukDef}. Suppose $W$ is  finite-dimensional, so that the $\li[1]$-algebra structure is encoded\footnote{The multibrackets on $W$ are recovered from $Q$ applying Thm. \ref{voronovderived} to the V-data $(L=\chi(W), \ga=\{\text{constant vector fields on $W$}\}, P(X)=X|_0, Q)$.} by a degree 1, self-commuting vector field $Q$ on $W$ \cite[Ex. 4.1]{Vor}. 
{We recall the following  fact, that holds for any vector field $X$ on $W_0$ and any element $m\in W_0$ (which defines a constant vector field $m$ on $W_0$):} 
\begin{equation}\label{eq:push}
X|_m=(e^{[m,\cdot]}X)|_0.
\end{equation}
{Indeed, both sides equal $((\phi_{-1})_*X)|_0$, where $\phi$ denotes the time one flow of $m$ (translation by $m$).
Eq. \eqref{eq:push} applied to $X=Q$ implies immediately that a point $m\in W_0$ is a zero of $Q$ if{f} $-m$ satisfies the Maurer-Cartan equation \eqref{MaCa}.}

View $z\in W_{-1}$ as a constant (degree $-1$) vector field on $W$. Then $[Q,z]$  is a degree zero vector field.
 As $\cL_{[Q,z]}Q=[[Q,z],Q]=0$, the flow of $[Q,z]$ preserves  the set of zeros of  $Q$, and hence $[Q,z]$ is tangent to this set. {Eq. \eqref{eq:push} applied to $X=[Q,Z]$ implies} that $[Q,z]|_{W_0}$ is the pushforward by $-Id_{W_0}$ of $\mathcal{Y}^z$, therefore $\mathcal{Y}^z$ is tangent to $MC(W)$.
 
 {A computation shows that 
$\cD$ can also be described in terms of all degree $-1$ vector fields: 
$\cD|_m=\{[Q,Z]|_m: Z\in \chi_{-1}(W)\}$ for all $m\in MC(W)$.
Since $[[Q,Z],[Q,Z']]=[Q,[[Q,Z],Z']]$ it follows that $\cD$ is involutive.}
\end{rem}
 
We will display explicitly the equivalence relation  induced on twisted Poisson structures in \S \ref{sec:tpois}, and show  that in this case the equivalence classes coincide with the orbits of a 
group action.


\section{Applications to Poisson geometry}\label{Poisson}

In this section we apply the machinery developed in \S \ref{sec:main} to examples arising from Poisson geometry. We study deformations of 
Poisson manifolds and coisotropic submanifolds in \S\ref{cois}. We consider deformations of 
Courant algebroids and Dirac structures in \S\ref{dirac}, focusing on the special case of twisted Poisson structures (and discussing equivalences) in \S\ref{sec:tpois}. Finally, we consider deformations of
{Courant algebroids and generalized complex structures in} \S\ref{sec:gcs}, discussing the case of complex structures in \S\ref{sec:cs}.

\subsection{Coisotropic submanifolds of Poisson manifolds}
 
\label{cois}

In this subsection we consider deformations of Poisson structures on a  manifold $M$ and deformations of coisotropic submanifolds. {We build on work of Oh and Park \cite{OP}, who realized that  deformations of a coisotropic submanifold of a fixed symplectic manifold are governed by a $L_{\infty}[1]$-algebra, and on work of Cattaneo and Felder 
\cite{CaFeCo2} who associate an $L_{\infty}[1]$-algebra to any coisotropic submanifold of a Poisson manifold.}

 Our main reference for this deformation problem is \cite[\S 3.2]{FloDiss}, which is based on \cite{OP} and \cite{CaFeCo2}. Recall that a \emph{Poisson structure} on $M$ is a bivector field $\pi$ on $M$ such that $[\pi,\pi]=0$, where the bracket denotes the Schouten bracket, and that a submanifold $C\subset (M,\pi)$ is \emph{coisotropic} if $\pi^{\sharp} TC^{\circ}\subset TC$, where $TC^{\circ}:=\{\xi\in T^*M|_C:   \xi|_{TC}=0\}$ and $\pi^{\sharp} 
 \colon T^*M\to  TM$ is the contraction with $\pi$ \cite{CW}.

Let $M$ be a manifold. Let $C\subset M$ be a submanifold. Fix an embedding of the normal bundle $\nu C:=TM|_C/TC$ into a tubular neighborhood of $C$ in $M$, such that the embedding {and its derivative} are the identity on $C$. 
In the following we will identify $\nu C$ with its image in $M$.

 We say that a vector field on $\nu C$ is \emph{fiberwise polynomial} if it preserves the fiberwise polynomial functions on the vector bundle $\nu C$. 
  {Such a vector field $X$ has \emph{polynomial degree $n$} (denoted $|X|_{pol}=n$) if  its action on fiberwise polynomial functions raises their degree (as polynomials) at most by $n$.   
  Locally,  choose local coordinates on $C$ and linear coordinates along the fibers of $\nu C$, which we denote collectively by $x$ and $p$ respectively. Then the {fiberwise polynomial} vector fields are exactly those which are sums of expressions $f_1(x)F_1(p) \pd{x}$ and $f_2(x)F_2(p)\pd{p}$ where $f_i\in C^{\infty}(C)$ and the $F_i$ are polynomials.
The   polynomial degrees of the two vector fields exhibited here are $deg(F_1)$ and $deg(F_2)-1$ respectively. }

Consider $\chi^{\bullet}(\nu C)$, the space of multivector fields on the total space $\nu C$, and denote by $\chi^{\bullet}_{fp}(\nu C)$ the sums of products of fiberwise polynomial vector fields. $\chi^{\bullet}(\nu C))[1]$ is a graded Lie algebra when endowed with the Schouten bracket $[\cdot,\cdot]$, and 
$\chi^{\bullet}_{fp}(\nu C)[1]$ is a graded Lie subalgebra. {The notion of polynomial degrees carries on to fiberwise polynomial  multivector fields, by $|X_1\wedge\dots\wedge X_k|_{pol}=\sum_i |X_i|_{pol}$. The Schouten bracket  preserves the polynomial degree (this is clear if we think of multivector fields as acting on tuples of functions).}

 Sections in $\Gamma(\wedge \nu C)$ 
can be regarded as elements of $\chi^{\bullet}_{fp}(\nu C)$ which are {vertical} (tangent to the fibers) and fiberwise constant.
 {A \emph{fiberwise polynomial Poisson} bivector field on $\nu C$ is an element $\pi \in  \chi^{2}_{fp}(\nu C)$ such that
 $[\pi,\pi]=0$. Notice that the associated Poisson bracket raises the degree of fiberwise polynomial functions on $\nu C$ by at most 
$|\pi|_{pol}$.}

\begin{rem}
The condition that a Poisson structure  be fiberwise polynomial is quite strong. 
The results of this subsection are extended in \cite{OPPois}  to  Poisson structures in a neighborhood $U\subset \nu C$ of the zero section which are ``fiberwise entire'', in the following sense: the Poisson bracket of two fiberwise polynomials functions, restricted to $U\cap \nu_xC$ , is given by a converging power series (for any $x\in C$).
\end{rem}

\begin{lem}\label{keycoiso} Let $\pi$ be  a {fiberwise polynomial} Poisson structure  on $\nu C$.
 The following quadruple forms a curved V-data:
\begin{itemize}
\item the graded Lie algebra $L:=\chi^{\bullet}_{fp}(\nu C)[1]$
\item its abelian subalgebra $\ga:=\Gamma(\wedge \nu C)[1]$
\item the natural projection $P \colon L \to \ga$ given by restriction to $C$ and projection along $\wedge T(\nu C)|_C\to \wedge \nu C$  
\item $\Delta:=\pi$,
\end{itemize}
hence by Thm. \ref{voronovderived} we obtain a  curved $L_{\infty}[1]$-structure  $\ga^P_{\Delta}$.

 {Its Maurer-Cartan equation reads
\begin{equation}\label{eq:mccoiso}
P\sum_{n=0}^{|\pi|_{pol}+2}\frac{1}{n!}[[\dots[\pi,\underbrace{\Phi],\dots],\Phi}_{n \text{ times }}]=0,
\end{equation}
where $\Phi\in \Gamma(\nu C)[1]$  is seen as a vertical vector field on $\nu C$.}
$\Phi\in \Gamma(\nu C)[1]$ is a Maurer-Cartan element in $\ga^P_{\Delta}$ if{f} ${graph(-\Phi)}$  is a coisotropic submanifold of $(\nu C, \pi)$.

Further, the above quadruple forms a   V-data if{f} $C$  is a coisotropic
submanifold of $(\nu C,\pi)$.
\end{lem}

\begin{proof} The fact that the above quadruple forms a curved V-data  is essentially the content of \cite[\S 2.6]{CaFeCo2}. For a more detailed proof we refer to 
\cite[Lemma 3.3 in \S 3.3]{FloDiss}, use that $\chi^{\bullet}_{fp}(\nu C)$ is a graded Lie subalgebra of $\chi^{\bullet}(\nu C)$, 
and use that $[\pi,\pi]=0$ by the definition of Poisson structure.

 {To prove eq. \eqref{eq:mccoiso}  we argue as follows.
Elements $a_i\in \ga_0=\Gamma(\nu C)[1]$, seen as vertical vector fields on $\nu C$, have {polynomial} degree $-1$ (in coordinates they read $f(x)\pd{p}$). Since  the Schouten bracket  preserves the polynomial degree, $[[\dots[\pi,a_1],\dots],a_{n}]$ has  polynomial degree  $|\pi|_{pol}-n$. Since the  polynomial degree of a  non-vanishing bivector field is $\ge -2$, we conclude that the above iterated brackets vanish for $n> |\pi|_{pol}+2$.}

 The equivalence\footnote{See  \cite[Ex. 3.2 in \S 4.3]{FloDiss} for an
example where $\pi$ is not fiberwise polynomial  and the correspondence fails.}
 between $\Phi\in \Gamma(\nu C)[1]$ being a Maurer-Cartan element  
  and $graph({-\Phi})$ being a coisotropic submanifold of $(\nu C, \pi)$   is proven as follows. 
Denote by $\psi \colon \nu C \to \nu C$ the time-1 flow of the vector field $\Phi$ (so $\psi$ is just translation by $\Phi$). In particular  $\psi(graph(-\Phi))=C$.
The pushforward bivector field by $\psi$ satisfies   $\psi_*(\pi)=e^{[\cdot,\Phi]}\pi$. Hence  
$graph({-\Phi})$ is coisotropic (w.r.t. $\pi$) if{f}
$C$ satisfies the  coisotropicity condition w.r.t.  $e^{[\cdot,\Phi]}\pi$, which is just
eq. \eqref{eq:mccoiso}.
To show that $\psi_*(\pi)=e^{[\cdot,\Phi]}\pi$, let 
$f,g$ be fiberwise polynomial functions on $\nu C$.
We have $\psi^*f=
e^{\Phi}f$ using the Taylor expansion of $f$ on each fiber. Hence 
\begin{align*}
&(\psi_*\pi)(f,g)=(\psi^{-1})^*(\pi(\psi^*f,\psi^*g))=e^{-{\Phi}}(\pi(e^{\Phi}f,e^{\Phi}g))=\\
&e^{[\cdot,{\Phi}]} [[\pi, e^{[{\Phi},\cdot]} f],e^{[{\Phi},\cdot]} g]=
[[e^{[\cdot,{\Phi}]} \pi, f], g]=(e^{[\cdot,{\Phi}]} \pi)( f, g).
\end{align*}
    
For the last statement, {use Thm. \ref{voronovderived}} and notice that $C$ is coisotropic if{f} we can write $\pi=\sum_j X_j\wedge Y_j$ with $X_j$ tangent to $C$, i.e. if{f} $\pi \in ker(P)$. 
\end{proof}

Hence we can apply  Thm. \ref{machine} (with $\Delta=\pi=0$ and $\Phi=0$):

\begin{cly}\label{cor:coiso} 
Let $C$ be a submanifold of a manifold, and consider a tubular neighborhood $\nu C$. For all $\tilde{\pi}\in \chi_{fp}^2(\nu C)$  
and 
 $\tilde{\Phi}\in \Gamma(\nu C)$:
\begin{align*}
&\begin{cases}
\tilde{\pi} \text{ is a Poisson structure   } \\
graph({-\tilde{\Phi}}) \text{ is a coisotropic submanifold of  } (\nu C, \tilde{\pi} )
\end{cases}\\
\Leftrightarrow
&
(\tilde{\pi}[2],\tilde{\Phi}[1]) \text{ is a MC element } \text{ of the \li[1]-algebra } \chi^{\bullet}_{fp}(\nu C)[2]\oplus \Gamma(\wedge \nu C)[1].
\end{align*} 
The above $\li[1]$-algebra structure is given by the  multibrackets
 (all other vanish)
\begin{align*}
d(X[1])&= PX,\\
\{X[1],Y[1]\}&=[X,Y][1](-1)^{|X|},\\
\{ X[1],a_1,\dots,a_n\}&=P[\dots[X,a_1],\dots,a_n]\;\;\;\;\;\; \text{for all }n\ge 1
\end{align*}
where  $X,Y\in \chi^{\bullet}_{fp}(\nu C)[1]$, $a_1,\dots,a_n \in \Gamma(\wedge \nu C)[1]$, and $[\cdot,\cdot]$ denotes the Schouten bracket on $\chi^{\bullet}_{fp} (\nu C)[1]$.
 \end{cly} 
 
 \begin{rem}
1)  The formulas for the multibrackets in Cor. \ref{cor:coiso}  show that the Maurer-Cartan equation for $(\tilde{\pi}[2],\tilde{\Phi}[1])$ has at most $|\tilde{\pi}|_{pol}+2$ terms, by the same argument as in Lemma \ref{keycoiso}.

2) It is known  that the deformation problem of coisotropic submanifolds in Poisson (even symplectic) manifolds is formally obstructed \cite{OP}. Cor. \ref{cor:coiso}  is used in
\cite{OPPois} to show that the same applies to the simultaneous deformation problem of coisotropic submanifolds  and fiberwise entire Poisson structures.
\end{rem}
 
{
We now display the $L_{\infty}[1]$-algebra governing the deformations of
a Poisson structure $\pi$ and of 
a coisotropic submanifold $C$.}

\begin{cly} 
{
Let $(M,\pi)$ be a Poisson manifold, 
$C$  a coisotropic  submanifold. Identify a tubular neighborhood of $C$ in $M$ with the normal bundle $\nu C$ in such a way that $\pi$ is fiber-wise polynomial. 
There is an    $L_{\infty}[1]$-algebra structure on
     $\chi^{\bullet}_{fp}(\nu C)[2]\oplus \Gamma(\wedge \nu C)[1]$
   whose Maurer-Cartan elements are exactly pairs $(\tilde{\pi}[2],\tilde{\Phi}[1])$ where $\tilde{\pi}\in \chi_{fp}^2(\nu C)$  and 
 $\tilde{\Phi}\in \Gamma(\nu C)$ are such that $\pi+\tilde{\pi}$ is a Poisson structure and ${graph(-\Phi)}$ is a coisotropic submanifold w.r.t. $\pi+\tilde{\pi}$.
}     
     
{
     Its non-vanishing multi-brackets $\{\dots \}^\pi$ are given as follows:
\begin{align*}
d^\pi (X[1])&= (-[\pi,X][1],P(X)),\\
d^\pi(a)&= (0,P([\pi,a])),\\
\{X[1],Y[1]\}^\pi_2&=(-1)^{|X|}[X,Y][1],\\
\{a_1,\dots,a_n\}^\pi_n&=P([\dots[\pi,a_1],\dots,a_n])\;\;\;\;\;\; \text{for all }n\ge 1,\\
\{X[1],a_1,\dots,a_n\}^\pi_{n+1}&=P([\dots[X,a_1],\dots,a_n])\;\;\;\;\;\; \text{for all }n\ge 1,
\end{align*}
where  $X,Y\in \chi^{\bullet}_{fp}(\nu C)[1]$ and $a_1,\dots,a_n \in \Gamma(\wedge \nu C)[1]$.
}
\end{cly}

\begin{proof}
{
By eq. \eqref{eq:MCtw}, in order to obtain an $L_{\infty}[1]$-algebra whose    Maurer-Cartan elements are those specified in the statement of the present corollary, we can twist  
 the $L_{\infty}[1]$-algebra of Corollary \ref{cor:coiso} using the Maurer-Cartan element $(\pi[2],0)$. Notice that $(\pi[2],0)$ is really a Maurer-Cartan element, since $\pi$ is a Poisson structure and $C$ is coisotropic w.r.t. $\pi$.
} 
 
{ 
The multi-brackets $\{\dots \}^\pi$ are computed as in eq. \eqref{eq:twb}. Notice that the particular form of the multibrackts appearing in Corollary \ref{cor:coiso} forces
all terms on the right hand side of eq. \eqref{eq:twb} to be zero, except possibly for the first two. 
}
\end{proof}

\subsection{Dirac structures and Courant algebroids}\label{dirac}

In this subsection we consider a Courant algebroid structure on a fixed vector bundle and  a Dirac subbundle $A$. We study
 deformations of the Courant algebroid structure (with the constraint
 that 
 the symmetric pairing remains unchanged), and of 
 the  Dirac subbundle $A$. {Deformations of Dirac subbundles within a fixed Courant algebroid were studied by Liu, Weinstein, Xu \cite{LWX}
 and by Bursztyn, Crainic, {\v{S}}evera \cite{QPois}.}
 We will make use of  facts from \cite[\S 3]{QPois}  and Roytenberg's \cite[\S 3]{QLie}
 \cite[\S 3]{DimaThesis}\cite{Dima}. We refer  to \cite[\S 1.4]{FloDiss} or \cite{AlbFlRio} for some basic facts on graded geometry.

Recall that a \emph{Courant algebroid} consists of a vector bundle $E\to M$ with a non-degenerate symmetric pairing on the fibers, a bilinear operation $\oo \cdot,\cdot \cc$ on $\Gamma(E)$, and a bundle map $\rho \colon E\to TM$ satisfying compatibility conditions, see for instance \cite[Def. 4.2]{Dima}. An example is  $TM\oplus T^*M$ with the natural pairing, $\oo X+\xi ,Y+\eta \cc:=[X,Y]+\cL_X\eta-\iota_Y d\xi$, and $\rho(X+\xi)=X$ (this is sometimes called the  standard Courant algebroid). A \emph{Dirac structure} is a subbundle   $L\subset E$ such that $L$ equals its orthogonal w.r.t. the pairing, and so that $\Gamma(L)$ is closed under $\oo \cdot,\cdot \cc$, see  \cite{Cou}. Examples of Dirac structures for the standard Courant algebroid are provided by graphs of closed 2-forms and of Poisson bivector fields.

Fix a Courant algebroid $E\to M$, {maximal isotropic subbundles $A$ and $K$ (not necessarily  involutive) so that $E=A\oplus K$ as a vector bundle.}
Identify  $K\cong A^*$ via the pairing on the fibers of $E$.
Consider the maps $$\Gamma(\wedge^2 A^*)\to \Gamma(A)\;,\; \eta_1 \wedge \eta_2 \mapsto pr_{A}(\oo(0,\eta_1),(0,\eta_2)\cc,$$
$${\Gamma(\wedge^2 A)\to \Gamma(A^*)\;,\; a_1 \wedge a_2 \mapsto pr_{A^*}(\oo(a_1,0),(a_2,0)\cc,}$$
and view them as   elements $\psi\in \Gamma(\wedge^3 A)$ {and
$\varphi\in \Gamma(\wedge^3 A^*)$ respectively.}
 Denote by $d_A$ the 
degree $1$ derivation of $\Gamma(\wedge A^*)$ given by  
{the bracket
$[\cdot,\cdot]_A:= pr_{A}(\oo \cdot,\cdot \cc|_A)$ on $\Gamma(A)$ and the bundle map $\rho|_{A}\colon A\to TM$.} Similarly denote by $d_{A^*}$ the 
degree $1$ derivation
of $\Gamma(\wedge A)$ given by the  
bracket $[\eta_1,\eta_2]_{A^*}:=pr_{A^*}(\oo(0,\eta_1),(0,\eta_2)\cc$ on $\Gamma(A^*)$ and the bundle map $\rho|_{A^*}\colon A^*\to TM$. The data given $\psi$, {$\varphi$},  $(A,[\cdot,\cdot]_A,\rho|_A)$, and    $(A^*,[\cdot,\cdot]_{A^*},\rho|_{A^*})$ forms a {\emph{proto-bialgebroid}}. 
 From these data one can reconstruct the Courant algebroid structure on $E$:   the bilinear operation is recovered  as
\begin{align}\label{CB}
&\oo (a_1,\eta_1),(a_2,\eta_2) \cc=\\
&\left( [a_1,a_2]_A+\cL_{\eta_1}a_2-\iota_{\eta_2}d_{A^*}a_1+\psi(\eta_1,\eta_2,\cdot)\;,\;[\eta_1,\eta_2]_{A^*}+\cL_{a_1}\eta_2-\iota_{a_2}d_{A}\eta_1{+\varphi(a_1,a_2,\cdot)}\right)\nonumber
\end{align}
and the anchor as $\rho_A+\rho_{A^*} \colon A\oplus A^* \to TM$  
{(\cite[\S 3.8]{DimaThesis}, see also \cite[\S 3.2]{YvetteAllThat}).}

Recall that Courant algebroids are in bijective correspondence with 
 degree 2 symplectic graded manifolds $\cM$ together with a degree 3 function $\Delta\in C(\cM)$ satisfying $\op  \Delta,\Delta\cp =0$ \cite[Thm. 4.5]{Dima}.
 (Here $\op  \cdot,\cdot\cp $ denotes the degree $-2$ Poisson bracket on $C(\cM)$ induced by the symplectic structure).
  The Courant algebroid $E$ corresponds to 
  $$\left( \cM:=T^*[2]A[1]\;,\; \Delta={-\varphi}+h_{d_A}+F^*(h_{d_{A^*}})-\psi \right)$$
 with the canonical symplectic structure, by \cite[Thm. 3.8.2]{DimaThesis}. 
 Here we view $\psi\in \Gamma(\wedge^3 A)$ {and
$\varphi\in \Gamma(\wedge^3 A^*)$}  as  elements of $C_3(\cM)$. Further $h_{d_A}\in C_3(\cM)$ is the fiber-wise linear function   induced by $d_A$, the function  
 $h_{d_{A^*}}\in C_3(T^*[2]A^*[1])$ is defined similarly, and $F \colon T^*[2]A[1] \to T^*[2]A^*[1]$ is the canonical symplectomorphism known as \emph{Legendre transformation} \cite[\S 3.4]{DimaThesis}. We denote by $pr$ the cotangent projection $\cM \to A[1]$.

\begin{lem}\label{keydirac} Fix a Courant algebroid $E\to M$,
{and maximal isotropic subbundles $A$ and $K$   so that $E=A\oplus K$ as a vector bundle.} 
  The following quadruple forms a {curved} V-data:
\begin{itemize}
\item the graded Lie algebra $L:=C(\cM)[2]$ with Lie bracket\footnote{$\op  \cdot,\cdot\cp $, as a bracket on $L$, has degree zero. Hence $(L,\op  \cdot,\cdot\cp )$ is a graded Lie algebra.}
$\op  \cdot,\cdot\cp $
\item its abelian subalgebra $\ga:=pr^*(C(A[1]))[2]\cong \Gamma(\wedge A^*)[2]$ 
\item the natural projection $P \colon L \to \ga$ given by evaluation on the base $A[1]$
\item $\Delta={-\varphi}+h_{d_A}+F^*(h_{d_{A^*}})-\psi$,
\end{itemize}
hence by Thm. \ref{voronovderived} we obtain a {curved}  $L_{\infty}[1]$-structure  $\ga^P_{\Delta}$. 
For every $\Phi\in \Gamma(\wedge^2 A^*)$ we have: $\Phi[2]$ is a MC element of  $\ga^P_{\Delta} $ if{f} $$graph(-\Phi):=\{ (X-\iota_X \Phi):X\in A\}\subset A\oplus A^*=E$$  is a Dirac structure.

{Further, the above quadruple forms a   V-data if{f} $A$  is a Dirac structure of $E$.}
\end{lem}

\begin{proof}
Since $\op  \cdot,\cdot\cp $ is the canonical Poisson bracket on the cotangent bundle, 
the cotangent fibers and the base $A[1]$ are Lagrangian submanifolds. Hence
$\ga$ is an abelian Lie subalgebra of $L$ and $ker(P)$, which consists of function on 
$T^*[2]A[1]$ vanishing on the base, is a Lie subalgebra. We have $\op  \Delta,\Delta\cp =0$ since $\Delta$ induces a Courant algebroid structure on  $A\oplus A^*$.  
Hence the the above quadruple is a {curved} V-data, and by  Thm.  \ref{voronovderived}   we obtain a {curved} $L_{\infty}[1]$-algebra structure $\ga^P_{\Delta}$.

We compute the Maurer-Cartan equation of $\ga^P_{\Delta}$. Let $\Phi \in \ga_0=\Gamma(\wedge^2 A^*)[2]$. {We have $\{\emptyset\}=P\Delta=-\varphi$. Notice that $-\varphi$ does not appear in the remaining terms of the Maurer-Cartan equation, since $\{-\varphi,\Phi\}=0$, for both entries belong to  the abelian subalgebra $\ga$.}
From
the  expression in coordinates 
for $F^*(h_{d_{A^*}})$ it follows that 
$\op  F^*(h_{d_{A^*}}),\Phi\cp $ and $\op  -\psi,\Phi\cp $ vanish  on the base $A[1]$.
So 
\begin{equation*}
P\op  \Delta,\Phi\cp =\op  h_{d_A},\Phi\cp =
{d_A}\Phi \in \Gamma(\wedge^3 A^*)
\end{equation*} 
where we used \cite[Lemma 3.3.1 1)]{DimaThesis}. 
Further
$\op  \op  h_{d_A},\Phi\cp ,\Phi\cp =0$ since both $\op  h_{d_A},\Phi\cp $ and $\Phi$ lie in the abelian Lie subalgebra $pr^*(C(A[1]))$, and in coordinates it is clear that $\op  \op  -\psi,\Phi\cp ,\Phi\cp $ vanishes on the base $A[1]$. So
$$P\op  \op  \Delta,\Phi\cp ,\Phi\cp =\op  \op  F^*(h_{d_{A^*}}),\Phi\cp ,\Phi\cp =-[\Phi,\Phi]_{A^*}$$ where we used  \cite[Lemma 3.6.2]{DimaThesis}.
Further, 
$$P\op  \op  \op  \Delta,\Phi\cp ,\Phi\cp ,\Phi\cp =\op  \op  \op  -\psi,\Phi\cp ,\Phi\cp ,\Phi\cp =
(\Phi^{\sharp}\wedge \Phi^{\sharp} \wedge \Phi^{\sharp})\psi
\in \Gamma(\wedge^3 A^*),$$
where $\Phi^{\sharp}  \colon A \to A^*, v \mapsto \iota_v\Phi$ is the contraction in the first component {(see \cite[\S 4.1.2]{YvetteAllThat})}.
All the other terms of the Maurer-Cartan equation   vanish. 
Hence we conclude that the Maurer-Cartan equation is
\begin{equation}\label{eq:MC4}
{-\varphi}+{d_A}\Phi-\frac{1}{2}[\Phi,\Phi]_{A^*}{+}\wedge^3\tilde{\Phi}(\psi)=0
\end{equation}
where $\wedge^3\tilde{\Phi}$ is defined as in \S \ref{sec:tpois}.

{We show that $\Phi$ satisfies eq. \eqref{eq:MC4} if{f}  $graph(-\Phi)$ is a Dirac structure, using the results of 
 \cite[\S 4]{QLie}. There Roytenberg  considers the time one flow $F_{\Phi}$ of the hamiltonian vector field of $\Phi$ on $\cM$. As $F_{\Phi}$ is a symplectomorphism of $\cM$, it corresponds to an isomorphism of Courant algebroids between $E'$ and $E$, where $E'$ is the Courant algebroid\footnote{Both Courant algebroids have the same underlying vector bundle and same symmetric pairing.} corresponding to the degree 3 function $F_{\Phi}^*(\Delta)$ on $\cM$.
The pullback function $F_{\Phi}^*(\Delta)$ splits in a sum according to 
bi-degree, and the component
lying in $\Gamma(\wedge^3 A^*)$ is
\begin{equation}\label{eq:comp}
 -\varphi-{d_A}\Phi-\frac{1}{2}[\Phi,\Phi]_{A^*}-\wedge^3\tilde{\Phi}(\psi),
\end{equation}
see  \cite[Eq. (4.2)]{QLie}.
$A$ is a Dirac structure in $E'$ if its image under the isomorphism $E'\cong E$, which is $graph(\Phi)$, is a Dirac structure in $E$.
On the other hand, as one sees easily from \cite[Thm. 3.8.2]{DimaThesis}, $A$ is a Dirac structure in $E'$ if{f} the component of $F_{\Phi}^*(\Delta)$  lying in $\Gamma(\wedge^3 A^*)$ (which is given by \eqref{eq:comp}) vanishes, or equivalently if   $-\Phi$ satisfies the Maurer-Cartan equation \eqref{eq:MC4}. Putting together these two statements proves the claim.}

{Finally, notice that
$\Delta\in ker(P)$ if{f} its component in the bi-degree corresponding to $\Gamma(\wedge^3 A^*)$, which is $-\varphi$, vanishes.
Using   again \cite[Thm. 3.8.2]{DimaThesis} 
we see that the  quadruple $(L,\ga,P,\Delta)$ forms a   V-data if{f} $A$  is a Dirac structure of $E$.
}
 \end{proof}

\begin{cly}\label{cor:dirac}
Fix a Courant algebroid $E\to M$,
a Dirac structure $A$, and a complementary isotropic subbundle $K$.  
Let $(L,\ga,P,\Delta)$ as in Lemma \ref{keydirac}
For all $\tilde{\Delta}\in C(\cM)_3$ 
and 
 $\tilde{\Phi}\in \Gamma(\wedge^2A^*)$:
\begin{align*}
&\begin{cases}
    \Delta+\tilde{\Delta} \text{ defines a new Courant algebroid }\\
    \text {structure on the vector bundle }E,\\ 
graph(-\tilde{\Phi}) \text{ is a Dirac structure  there } 
\end{cases}
\Leftrightarrow 
&
(\tilde{\Delta}[3],\tilde{\Phi}[2]) \in MC\big(({L}[1]\oplus \ga)_\Delta^P\big).
\end{align*}
  \end{cly}
\begin{proof}
{$(L,\ga,P,\Delta)$ is a V-data by the last statement of Lemma \ref{keydirac}, hence we can apply
 Thm. \ref{machine} with $\Phi=0$. Use again Lemma \ref{keydirac} 
 to phrase the conclusions of Thm.  \ref{machine} in terms of Courant algebroids and Dirac structures.}
\end{proof}

\begin{rem}\label{rem:convDirac} 
 We check that {the V-data $(L,\ga, P,\Delta)$ is filtered
(Def. \ref{triple}).} $T^*[2]A^*[1]$ is a vector bundle over $A^*[1]$, so we can denote by $C^k(T^*[2]A^*[1])$ the functions which are polynomials of degree $k$ on each fiber. Using the Legendre transformation $F$ to identify $\cM=T^*[2]A[1]$ with $T^*[2]A^*[1]$ we obtain a direct product decomposition $L=\prod_{k \ge -1}L^k$ where
$L^k:=C^{k+1}(T^*[2]A^*[1])$. Notice that an element of $pr^*(C_{k+1}(A[1]))[2]\cong \Gamma(\wedge^{k+1} A^*)[2]$ lies in $L^k$.  
By Remark \ref{rem:dirprod}, $\cF^nL:=\prod_{k\ge n}L^k$ is   a complete filtration of the vector space $L$. The remaining items of Def. \ref{triple} are easily checked.
 \end{rem}

\subsection{Twisted Poisson structures}\label{sec:tpois}

In this subsection we present a special case of the situation studied in \S \ref{dirac}. We
apply  Cor. \ref{cor:dirac} to the standard Courant algebroid over a manifold $M$  and $A=T^*M$. We obtain a $\li[1]$-algebra whose  Maurer-Cartan elements consist of closed 3-forms and twisted Poisson structures \cite{SW}, recovering the $\li[1]$-algebra recently displayed by Getzler \cite{GetRio}, and study their equivalences. {Further, given a closed 3-form $H$ and an $H$-twisted Poisson structure, we describe the $\li[1]$-algebra governing the deformations of the pair $(H,\pi)$.}
Twisted Poisson structures appeared in relation to deformations also in \cite[\S 3]{Park}.

We will need the following notation: 
for $\pi \in \wedge^{a}TM$ and $a\ge 1$ we define
$$\pi^{\sharp} \colon T^*M\to \wedge^{a-1}TM\;\;,\;\; \xi \to \iota_\xi \pi,$$
and we define $\pi^{\sharp}\equiv 0$ if $a=0$.
We also need an extension of the above to several multivectors: for 
$\pi_1 \in \wedge^{a_1}TM,\dots,\pi_n \in \wedge^{a_n}TM$ ($n\ge 1, a_i\ge 1$), we define
\begin{align*}
\pi_1^{\sharp}\wedge \dots \wedge \pi_n^{\sharp}\;\; \colon \;\;\;\wedge^nT^*M &\to \wedge^{a_1+\cdots+a_n-n}TM,\\
\xi_1\wedge\dots\wedge\xi_n &\mapsto \sum_{\sigma\in S_n} (-1)^{\sigma}\pi_1^{\sharp}(\xi_{\sigma(1)})\wedge\dots\wedge \pi_n^{\sharp}(\xi_{\sigma(n)})\end{align*}
where   $\xi_i\in T^*M$ and $(-1)^{\sigma}$ is the sign of the permutation $\sigma$.

Recall that, given a bivector field $\pi$ and a closed 3-form $H$, one says that  $\pi$ is a \emph{$H$-twisted Poisson structure} \cite[eq. (1)]{SW} if{f} $$[\pi,\pi]=2 \wedge^3\tilde{\pi} (H),$$   where $\wedge^3\tilde{\pi}=\frac{1}{6}(\pi^{\sharp}\wedge \pi^{\sharp}\wedge \pi^{\sharp})$ and $[\cdot,\cdot]$ is the Schouten bracket of multivectorfields.

\begin{cly}\label{cor:tpois}
Let $M$ be a manifold. There is an $\li[1]$-algebra structure on $$\mathfrak{L}:=\Omega^{\bullet\ge 1} (M)[3]\oplus \chi^{\bullet}(M)[2]$$ whose only non-vanishing multibrackets are
\begin{itemize}
\item[a)] minus the   de Rham differential  on differential forms,
  \item[b)] $\{\pi_1,\pi_2\}=[\pi_1,\pi_2] (-1)^{a_1+1}$, where  $\pi_i \in \chi^{a_i}(M)$,\item[c)] $\{H,\pi_1,\dots,\pi_n\}\;=\;(-1)^{\sum_{i=1}^n a_i(n-i)}(\pi_1^{\sharp}\wedge \dots \wedge \pi_n^{\sharp})H$ \\for all $n\ge1$, where $H\in\Omega^n(M)$ and $\pi_1 \in \chi^{a_1}(M),\dots,\pi_n \in \chi^{a_n}(M)$ with all $a_i\ge1$.
\end{itemize}
Its Maurer-Cartan elements are exactly pairs $(H[3],\pi[2])$ where $H\in \Omega^3(M)$ and $\pi \in \chi^2(M)$ are such that $dH=0$ and $\pi$ is a $H$-twisted Poisson structure.  \end{cly}
\begin{rem}
The graded vector space $\mathfrak{L}= \Omega^{\bullet\ge 1} (M)[3]\oplus \chi^{\bullet}(M)[2]$ is concentrated in degrees $\{-2,\dots,dim(M)-2\}$, and its degree $i$ component is $\Omega^{i+3}(M)\oplus \chi^{i+2}(M)$.
\end{rem}

\begin{proof}
We apply  Cor. \ref{cor:dirac} to the standard Courant algebroid $TM\oplus T^*M$ (defined at the beginning of \S \ref{dirac}), to $A=T^*M$ and $K=TM$. Notice that it corresponds to the Lie bialgebroid $(A,K)$, where $A$ has the zero structure and $K=TM$ has its canonical Lie algebroid structure.

 We use the following notation for the canonical local coordinates on $\cM:=T^*[2]T^*[1]M$: we denote by 
$x_j$ arbitrary local coordinates on $M$, by $p_j$ the canonical coordinates on the fibers of $T^*[1]M$ (so the degrees are 
$|x_j|=0, |p_j|=1$, for $j=1,\dots,dim(M)$). 
By $P_j,v_j$ we denote the conjugate coordinates on the fibres of $\cM\to T^*[1]M$, with degrees   $|P_j|=2, |v_j|=1$. One has $\op P_j,x_k\cp=\delta_{jk}$ and $\op p_j,v_k\cp=\delta_{jk}$. The element of $C_3(\cM)$ corresponding to the standard Courant algebroid is $\cS:=\sum_i P_iv_i$.

The quadruple appearing in Lemma \ref{keydirac} reads
\begin{itemize}
\item   $L:=C(T^*[2]T^*[1]M)[2]$, whose Lie bracket we denote by  $\op\cdot,\cdot\cp$
\item   $\ga:= C(T^*[1]M))[2]\cong \chi^{\bullet}(M)[2]$ 
\item the natural projection $P \colon L \to \ga$ given by evaluation on the base $T^*[1]M$, i.e. setting $P_j=0,v_j=0$ for all $j$
\item $\Delta=\sum_i P_iv_i$.
\end{itemize}
The multibrackets of the $\li[1]$-algebra $(L[1]\oplus \ga)^P_\Delta$ are given in Thm. \ref{voronov}.
Notice that using the Legendre transformation $F$ we have $$\Omega(M)[2]=C(T[1]M)[2]\subset C(T^*[2]T[1]M)[2] {\cong}L,$$ and  
$\Omega^{\bullet\ge 1} (M)[2]\subset ker(P)$ is a Lie subalgebra preserved by $\op\Delta,\cdot\cp$. So by Remark \ref{kerPa} it follows that  $\mathfrak{L}=\Omega^{\bullet\ge 1} (M)[3]\oplus \chi^{\bullet}(M)[2]$ is a $\li[1]$-subalgebra of $(L[1]\oplus \ga)^P_\Delta$. We justify why the restriction of the  multibrackets  to $\mathfrak{L}$ is the one described in the statement of this corollary.   a)   follows from eq. \eqref{diffV2} and   $$\op \sum_i P_iv_i,
F(x)v_{\epsilon(1)}\dots v_{\epsilon(k)}\cp=\sum_i\frac{\partial F}{\partial x_i}v_i v_{\epsilon(1)}\dots v_{\epsilon(k)},$$ where $\epsilon(i)=1,\dots,dim(M)$.    b)  follows from  eq.  \eqref{vorderlong} and  \cite[Lemma 3.6.2]{DimaThesis}.
 c)   follow from eq. \eqref{vorder} and a lengthy but straightforward computation in coordinates.

For the statement on Maurer-Cartan elements we proceed as follows. Given $H\in \Omega^3(M)$, the degree $3$ function $\sum_i P_iv_i+H$ on $\cM$ defines a Courant algebroid structure (i.e., is self-commuting) if{f} $H$ is closed, and in this case it
induces the $(-H)$-twisted\footnote{Recall that the
  $K$-twisted Courant algebroid is $TM\oplus T^*M$  with bilinear operation  $\oo X+\xi, Y+\eta \cc_K:=[X,Y]+\cL_X\eta-\iota_Y d\xi+\iota_Y\iota_X K$.} Courant algebroid $(TM\oplus T^*M)_{-H}$  \cite[\S4]{Dima}\cite[\S 8]{HDirac}. 
Hence, by Cor. \ref{cor:dirac}, $(H[3],\pi[2])$ is a Maurer-Cartan element of  $\mathfrak{L}$ if{f} $H$ is closed and $graph(-\pi)$ is a Dirac structure in $(TM\oplus T^*M)_{-H}$. The latter condition is equivalent to 
$-\pi$ being  a $(-H)$-twisted Poisson structure  \cite[\S 3]{SW}, that is, to $\pi$ being a $H$-twisted Poisson structure. 
\end{proof}

{
Given a closed 3-form $H$ and an $H$-twisted Poisson structure $\pi$,
we now describe the $L_{\infty}[1]$-algebra governing deformations of the pair $(H,\pi)$ to   pairs consisting of a closed 3-form and a correspondingly twisted Poisson structure. 
}

\begin{cly}\label{cor:tpoisany}
{
Let $M$ be a manifold,
$H$ a closed 3-form  and $\pi$ an $H$-twisted Poisson structure.
There is an $\li[1]$-algebra structure on $$\Omega^{\bullet\ge 1} (M)[3]\oplus \chi^{\bullet}(M)[2]$$
whose Maurer-Cartan elements are exactly pairs $(\widetilde{H}[3],\widetilde{\pi}[2])$, where $\widetilde{H}\in \Omega^3(M)$ and $\widetilde{\pi} \in \chi^2(M)$ are such that $\widetilde{H}$ is closed and $\pi+\widetilde{\pi}$ is an $(H+\widetilde{H})$-twisted Poisson structure.
}

{
Its $n$-th multibracket is 
$\{\dots\}_n+\{\dots\}_n'$, where $\{\dots\}_n$ denotes the $n$-th multibracket of the $\li[1]$-algebra   
described in Cor. \ref{cor:tpois} and $\{\dots\}_n'$ is defined as follows:
}

\begin{itemize}

\item[i)] 
{
for $n=1$:
$$\{(\widetilde{H}_1,\widetilde{\pi}_1)\}'_1=-[\pi,\widetilde{\pi}_1]
+\frac{1}{2}\{H,\pi,\pi,\widetilde{\pi}_1\}_4+
\frac{1}{\widetilde{n}_1!}\{\widetilde{H}_1,
\pi,\dots,\pi
\}_{\widetilde{n}_1+1}$$
}

  \item[ii)] 
{  
  for $n=2$:
   \begin{align*}
\{(\widetilde{H}_1,\widetilde{\pi}_1),(\widetilde{H}_2,\widetilde{\pi}_2)\}'_2=\{H,\pi,\widetilde{\pi}_1,\widetilde{\pi}_2\}_4
+ \delta_{\widetilde{n}_1,\ge 2}
&\{\widetilde{H}_1,
\pi,\dots,\pi,
\widetilde{\pi}_2\}_{\widetilde{n}_1+1}\frac{1}{(\widetilde{n}_1-1)!}\\
+ \delta_{\widetilde{n}_2,\ge 2}
&\{\widetilde{H}_2,
\pi,\dots,\pi,
\widetilde{\pi}_1\}_{\widetilde{n}_2+1}\frac{(-1)^{(\widetilde{n}_2-1)\widetilde{a}_1}}
{(\widetilde{n}_2-1)!}
\end{align*} 
}

\item[iii)] 
{
for $n=3$:
\begin{align*}
\{(\widetilde{H}_1,\widetilde{\pi}_1),(\widetilde{H}_2,\widetilde{\pi}_2),(\widetilde{H}_3,\widetilde{\pi}_3)\}'_3=
\{H,\widetilde{\pi}_1,\widetilde{\pi}_2,\widetilde{\pi}_3\}_4
+ 
\delta_{\widetilde{n}_1,\ge 3}
&\{\widetilde{H}_1,
\pi,\dots,\pi,
\widetilde{\pi}_2,\widetilde{\pi}_3\}_{\widetilde{n}_1+1}\frac{1}{(\widetilde{n}_1-2)!}\\
+ \delta_{\widetilde{n}_2,\ge 3}
&\{\widetilde{H}_2,\pi,\dots,\pi,
\widetilde{\pi}_1,\widetilde{\pi}_3\}_{\widetilde{n}_2+1}\frac{(-1)^{(\widetilde{n}_2-1)\widetilde{a}_1}}
{(\widetilde{n}_2-2)!}\\
+ \delta_{\widetilde{n}_3,\ge 3}
&\{\widetilde{H}_3,\pi,\dots,\pi,
\widetilde{\pi}_1,\widetilde{\pi}_2\}_{\widetilde{n}_3+1}\frac{(-1)^{(\widetilde{n}_3-1)(\widetilde{a}_1+\widetilde{a}_2)}}
{(\widetilde{n}_3-2)!}
\end{align*}
}

\item [iv)] 
{
for $n\ge 4$:
$$\{(\widetilde{H}_1,\widetilde{\pi}_1),\dots,(\widetilde{H}_n,\widetilde{\pi}_n)\}'_n=
  \sum_{\substack{1\le i\le n \\ \widetilde{n}_i \ge n}}
\{\widetilde{H}_i,
\pi,\dots,\pi,
\widetilde{\pi}_1,\dots,\widehat{\widetilde{\pi}}_i,\dots,\widetilde{\pi}_n\}_{\widetilde{n}_i+1}
\frac{ (-1)^{(\widetilde{n}_i-1)
 (\widetilde{a}_1 +\dots+\widetilde{a}_{i-1})}}
{(\widetilde{n}_i-n+1)!}
$$
}
\end{itemize}
{
where  $(\widetilde{H}_i,\widetilde{\pi}_i)\in\Omega^{\widetilde{n}_i}(M)\oplus\chi^{\widetilde{a}_i}(M)$, $\delta$ is the Kronecker delta (so $\delta_{a,\ge b}=1$ if $a\ge b$ and zero otherwise), and $\widehat{\widetilde{\pi}}_i$ denotes omission of the element $\widetilde{\pi}_i$. 
}
 \end{cly}

{
\begin{proof}
By eq. \eqref{eq:MCtw}, the seeked $L_{\infty}[1]$-algebra is obtained twisting the $L_{\infty}[1]$-algebra
of Cor. \ref{cor:tpois} by the Maurer-Cartan element $(H[3],\pi[2])$.
The $n$-th twisted multibracket  is computed as in eq. \eqref{eq:twb}, and we write it as $\{\dots\}_n+\{\dots\}_n'$. Notice that of the three types of multibrackets defined in  Cor. \ref{cor:tpois}, type a) never appears  while computing $\{\dots\}_n'$, and type b) appears only when $n=1$. Notice further that type c) involves exactly one differential form and as many multivector fields as the degree of the form. (This explains for instance why $H$ does not appear in the  expression for $\{\dots\}_n'$ when $n\ge 4$.)
\end{proof}
}

\subsubsection{Equivalences of twisted Poisson structures}
 
Consider the $\li[1]$-algebra  $\mathfrak{L}$ of  Cor. \ref{cor:tpois}.  We make explicit the equivalence relation induced on its set of Maurer-Cartan elements.

  Fix $(B,X)\in \mathfrak{L}_{-1}=\Omega^2(M)\oplus \chi(M)$. 
It defines a vector field $\mathcal{Y}^{(B,X)}$ on $\mathfrak{L}_{0}=\Omega^3(M)\oplus \chi^2(M)$. By eq. \eqref{gaugeaction} and Cor. \ref{cor:tpois},
 at the point $(H,\pi)$ the vector field reads  \begin{align}\label{gaugetpois}
\mathcal{Y}^{(B,X)}|_{(H,\pi)}= \big(-dB\;,\; [X,\pi]+\wedge^2\tilde{\pi}(B-\iota_X H)\big) \end{align}
where $\wedge^2\tilde{\pi}:=\frac{1}{2}(\pi^{\sharp}\wedge \pi^{\sharp})$. 

 \begin{rem}
The binary bracket on $\mathfrak{L}_{-1}$ reduces to the Lie bracket of vector fields on $\chi(M)$, making $\mathfrak{L}_{-1}$ into a Lie algebra. The assignment $
(B,H)\mapsto  \mathcal{Y}^{(B,X)}$ is not a Lie algebra morphism. For instance, the bracket of $\mathcal{Y}^{(0,X)}$ and $\mathcal{Y}^{(0,\tilde{X})}$ differs from  
$\mathcal{Y}^{(0,[X,\tilde{X}])}$ for generic $X,\tilde{X}\in \chi(M)$, as one can check using the proof of Prop. \ref{cor:Gu} below.
\end{rem}

For any diffeomorphism $\phi$ of $M$, we consider the vector bundle automorphism $$TM\oplus T^*M,\; Y+\eta\mapsto \phi_*Y+(\phi^{-1})^*\eta,$$ which by abuse of notation we denote by $\phi_*$. For any $B\in \Omega^2(M)$, we consider $$e^B \colon 
TM\oplus T^*M,\; Y+\eta\mapsto Y+(\eta+ \iota_Y B).$$ Recall that the vector bundle $TM\oplus T^*M$ is endowed with a canonical pairing on the fibers given by $\langle  X_1+\xi_1, X_2+\xi_2 \rangle =\frac{1}{2}(\iota_{X_1}\xi_2+\iota_{X_2}\xi_1)$.
\begin{rem}\label{ebphi}
The group of vector bundle automorphisms of $TM\oplus T^*M$ preserving the canonical pairing and preserving\footnote{In the sense that the projection $TM\oplus T^*M \to TM$ is equivariant w.r.t. the vector bundle automorphism and the derivative of its base map.}
  the canonical projection $TM\oplus T^*M \to TM$ is given exactly by $\{\phi_* e^B: \phi\in \text{Diff}(M), B\in \Omega^2(M)\}$. This follows by the same argument as for \cite[Prop. 2.5]{Gu2}. 
Further notice that  $e^B \phi_*=\phi_*e^{\phi^* B}$.
\end{rem}
 Abusing notation, for any bivector field $\pi$ such that
 $1+B^{\flat}\pi^{\sharp} \colon T^*M \to T^*M$ is invertible, we denote by $e^B \pi$ the unique bivector field whose graph is $e^B(graph(\pi))$. Here $B^{\flat}$
is the contraction in the first component of $B$.

Consider the connected group $\Omega^2(M) \rtimes \text{Diff}(M)_{\circ}$ (where the second factor denotes the diffeomorphisms isotopic to the identity), with multiplication $$(B_1,\phi_1)\cdot (B_2,\phi_2)=(B_1+(\phi_1^{-1})^* B_2\;,\;\phi_1\circ \phi_2).$$
The   partial\footnote{The action is defined whenever $1+B^{\flat}(\phi_*\pi)^{\sharp}$ is invertible.}
 action of $\Omega^2(M) \rtimes \text{Diff}(M)_{\circ}$ on $\Omega^3(M)\oplus\chi^2(M)$  
by
\begin{equation*}
 (B,\phi)\cdot (H, \pi)=((\phi^{-1})^*(H)-dB\;,\; e^B \phi_* \pi)
\end{equation*}
preserves $$MC(\mathfrak{L})=\{(H,\pi)\in \Omega_{closed}^3(M)\oplus \chi^2(M): \pi \text{ is a }H\text{-twisted Poisson structure}\}. $$
This can easily be checked  using the following two facts.
First, for every $H\in H^3_{closed}(M)$, the isomorphism $e^B \phi_*\colon TM\oplus T^*M \to TM\oplus T^*M$  maps the $H$-twisted Courant bracket to the $(\phi^{-1})^*(H)-dB$-twisted Courant bracket \cite[\S 2.2]{Gu2}. 
Second, $\pi$ is a
 $H$-twisted Poisson structure if{f} $graph(\pi)$ is involutive w.r.t. the $H$-twisted Courant bracket.

\begin{rem}
The map $e^B \phi_*$ is an isomorphism of Courant algebroids (from the
$H$-twisted  to the $(\phi^{-1})^*(H)-dB$-twisted Courant algebroid) 
covering $\phi$, by the above and Rem. \ref{ebphi}. As a consequence,  twisted Poisson structures lying in the same orbit of the $\Omega^2(M) \rtimes \text{Diff}(M)_{\circ}$ action
share  many properties. For instance, their associated Lie algebroids are isomorphic, and in particular the underlying (singular) foliations  are diffeomorphic. 

Notice further that the $\Omega^2(M) \rtimes \text{Diff}(M)_{\circ}$ action on $MC(\mathfrak{L})$ preserves the cohomology class of closed 3-forms. For instance, the orbit through $(H=0,\pi=0)$ is $\{(H',0): H' \text{ is exact}\}$.
\end{rem}
  
We will show that the natural equivalence relation on $MC(\mathfrak{L})$
is given by the above $\Omega^2(M) \rtimes \text{Diff}(M)_{\circ}$ action.
To do so, we first need a technical lemma. 
\begin{lem}\label{ddtp}
Let $X$ be a vector field on a manifold $M$ with flow $\phi^t$ defined for $t\in I\subset \RR$, let $\{C_t\}_{t\in I}$ be a smooth family of 2-forms and let $\pi$ be a bivector field. Denote $\pi_t:=(\phi_t)_*(e^{C_t}\pi)$.
Then
\begin{equation}\label{eq:2terms}
\frac{d}{dt}\pi_t=[X,\pi_t]+\wedge^2\tilde{\pi}_t\left ((\phi_{-t})^*(\frac{d}{dt}C_t)\right).
\end{equation} 
\end{lem} 
\begin{proof}
We have
\begin{equation}\label{ddte}
\frac{d}{dt}(e^{C_t}\pi) =
\wedge^2\widetilde{(e^{C_t}\pi)} (\frac{d}{dt}C_t).
\end{equation}
This follows from  
$(e^{C_t}\pi)^{\sharp}=\pi^{\sharp}(1+C_t^{\flat}\pi^{\sharp})^{-1}$
\cite[\S 4]{SW}, and from $\frac{d}{dt}(e^{C_t}\pi)^{\sharp}=-
(e^{C_t}\pi)^{\sharp}(\frac{d}{dt}C_t)^{\flat}(e^{C_t}\pi)^{\sharp}$.
Using eq. \eqref{ddte} in the first equality we obtain
 \begin{align*}
\frac{d}{dt}\pi_t&=(\phi_t)_*\left( \frac{d}{dt}(e^{C_t}\pi)+[X,e^{C_t}\pi]\right)
\\
&=(\phi_t)_*\left( \wedge^2\widetilde{(e^{C_t}\pi)}(\frac{d}{dt}C_t)\right)+ 
(\phi_t)_*[X,e^{C_t}\pi]
\end{align*}
which equals the r.h.s. of eq. \eqref{eq:2terms}.
\end{proof}

\begin{prop}\label{cor:Gu} The leaves of the  involutive singular distribution 
\begin{equation}\label{distrtwp}
span\{\mathcal{Y}^{(B,X)}: (B,X)\in \mathfrak{L}_{-1}=\Omega^2(M)\oplus \chi(M)\}
\end{equation}
on $MC(\mathfrak{L})$ coincide   with the orbits of the 
  partial action of $\Omega^2(M) \rtimes \text{Diff}(M)_{\circ}$ on  $MC(\mathfrak{L})$.
 \end{prop}
\begin{proof} It suffices to show that \eqref{distrtwp} coincides with the singular distribution given by the infinitesimal action associated to the group action of 
$\Omega^2(M) \rtimes \text{Diff}(M)_{\circ}$. Notice that the Lie algebra of this group is $\Omega^2(M)\oplus \chi(M)$, so take an element $(B,X)\in  \Omega^2(M)\oplus \chi(M)$. We compute the corresponding generator of the action $\mathcal{Z}^{(B,X)}$ at a point
 $(H,\pi)\in MC(\mathfrak{L})$: we have
\begin{align}
\mathcal{Z}^{(B,X)}|_{(H,\pi)}:=\frac{d}{dt}|_{t=0}(tB,\phi_t)\cdot (H,\pi)=\big(-d(\iota_X H+B)\;,\; [X,\pi]+\wedge^2\tilde{\pi}(B) \big)
\end{align}
where $\phi_t$ is the flow of $X$ and  where we 
use Lemma \ref{ddtp} to compute  $\frac{d}{dt}|_{t=0} (\phi_t)_*e^{(\phi_t)^*(tB)}(\pi)$. 
Comparing this with eq. \eqref{gaugetpois} we see that 
$$\mathcal{Z}^{(B-\iota_X H,X)}|_{(H,\pi)}=\mathcal{Y}^{(B,X)}|_{(H,\pi)}.$$
This shows that the two singular distributions agree at the point $(H,\pi)$, and repeating at every point of $MC(\mathfrak{L})$ we conclude that the two   singular distributions agree on $MC(\mathfrak{L})$.
\end{proof}

We conclude describing explicitly the flow on $MC(\mathfrak{L})$ induced by a fixed element $(B,X)$ of $\mathfrak{L}_{-1}$.

\begin{prop}\label{flowH} Let $(B,X)\in \Omega^2(M)\oplus \chi(M)$. The integral curve of  $\mathcal{Y}^{(B,X)}$ starting at the point $(H,\pi)\in \mathfrak{L}_{0}$ reads
\begin{equation}\label{curveL0}
t \mapsto (H-tdB,(\phi_t)_*e^{C_t^H} \pi)
\end{equation}
where 
$\phi$ denotes the flow of $X$ and
\begin{equation*}
C_t^H:=D_t+\int_{0}^t (\phi_{s}^*)(B-\iota_X H) ds 
\end{equation*}
for $D_t$ the unique solution with $D_0=0$ of 
\begin{equation*}
\frac{d}{dt}{D_t}=t(\phi_{t}^*)\iota_X dB.
\end{equation*} 
(The above curve is defined as long as $\phi_t$ is defined and
 $1+(C_t^H)^{\flat}\pi^{\sharp}$ is invertible.) 
\end{prop}
\begin{proof}
Fix $(H,\pi)\in\mathfrak{L}_{0}$ and consider the curve defined in eq. \eqref{curveL0}.
The curve is tangent to the vector field $\mathcal{Y}^{(B,X)}$ at all times $t$, by virtue of  Lemma \ref{ddtp} and since 
$$(\phi_{-t})^*(\frac{d}{dt}C_t^H)=(\phi_{-t})^*[
t(\phi_{t}^*)\iota_X dB+(\phi_{t}^*)(B-\iota_X H)]=
B-\iota_X (H-tdB).$$
Since at time $t=0$ the curve is located at the point $(H,\pi)$, we are done.
\end{proof} 

\begin{rem}
   Let $(B,X)\in \Omega^2(M)\oplus \chi(M)$ where $B$ is closed, and let $(H,\pi)\in  \mathfrak{L}_0$. 
    Then $D_t=0$, and consequently 
$(\phi_t)_*e^{C_t^H}$ is a one parameter \emph{group} of orthogonal vector bundle automorphisms of $TM\oplus T^*M$ (see \cite[Prop. 2.6]{Gu2}). 
Hence 
the second component of  integral curve of $\mathcal{Y}^{(B,X)}$ starting at $(H,\pi)$ is    the image of (the graph of) $\pi$ under a  one parameter  group of orthogonal vector bundle automorphisms of $TM\oplus T^*M$.
\end{rem}

\subsection{Generalized complex structures and Courant algebroids}\label{sec:gcs}

In this subsection we
consider deformations of Courant algebroid structures on a fixed pseudo-Riemannian vector bundle and of their generalized complex structures. {Deformations of generalized complex structures within a fixed Courant algebroid were studied by Gualtieri in \cite[\S5]{Gu2}.}

Fix a Courant algebroid $E\to M$ and a generalized {almost} complex structure  $J$, i.e. a vector bundle map  $J \colon E \to E$ with $J^2=-Id$ preserving the fiberwise pairing.
  $J$ can be equivalently encoded by a  complex
 {maximal isotropic subbundle} 
   $A\subset E\otimes \CC$ transverse to the complex conjugate $\bar{A}$. The correspondence is as follows: given $J$, define $A$ to be the $+i$-eigenbundle of the complexification of $J$. Given $A$, consider the complex endomorphism of $E\otimes \CC$ with  $+i$-eigenbundle $A$ and  $-i$-eigenbundle $\bar{A}$, and define $J$ to be the restriction to $E$. {Further we have: $J$ is a \emph{generalized   complex structure} (i.e., it satisfies a certain integrability condition \cite{Hi}\cite[Def. 3.1]{Gu2}) if{f} $A$ is a \emph{complex Dirac
structure}.}

Hence we are in the situation of \S \ref{dirac}, except that we consider \emph{complex} 
 {maximal isotropic subbundles} 
in the complexification $E\otimes \CC$ of a (real) Courant algebroid. 
Notice that $E$ does not have a preferred splitting into 
 {maximal isotropic subbundles.} 
On the other hand, $E\otimes \CC$ is a complex Courant algebroid with a splitting $E\otimes \CC=A\oplus \bar{A}$ into complex 
 {maximal isotropic subbundles.} 
The construction of  \cite[Thm. 4.5]{Dima} leads to a \emph{complex} 
graded manifold\footnote{It is given by a sheaf of graded commutative algebras over $\CC$ satisfying the usual locally triviality condition.} with a degree 2 symplectic structure $\op\cdot,\cdot\cp$, namely $\cN=T^*[2]A[1]$. We denote its ``global functions'', a graded commutative algebra over $\CC$, by $C_{\CC}(\cN)$.

\begin{lem}\label{keygcs} 
Fix a Courant algebroid $E\to M$ and a generalized {almost} complex structure $J$, encoded by a complex 
{maximal isotropic subbundle}
$A$ transverse to $\bar{A}$.  The following quadruple forms a {curved} V-data:
\begin{itemize}
\item the complex graded Lie algebra $L:=C_{\CC}(\cN)[2]$ with Lie bracket $\op\cdot,\cdot \cp$
\item its complex abelian subalgebra $\ga:=pr^*(C_{\CC}(A[1]))[2] \cong \Gamma(\wedge A^*)[2]$ 
\item the natural projection $P \colon L \to \ga$ given by evaluation on the base $A[1]$
\item $\Delta={-\varphi}+h_{d_A}+F^*(h_{d_{A^*}}){-\psi}$,
 defined analogously to  \S \ref{dirac},
\end{itemize}
hence by Thm. \ref{voronovderived} we obtain a   complex\footnote{Hence the underlying graded vector space is complex and the multibrackets are $\CC$-linear.} {curved} $L_{\infty}[1]$-structure  $\ga^P_{\Delta}$.

For all $\Phi\in \Gamma(\wedge^2 A^*)$ we have: $\Phi[2]$ is a Maurer-Cartan element in $\ga^P_{\Delta}$  if{f} $$graph(-\Phi):=\{(X-\iota_X \Phi):X\in A\}\subset A\oplus \bar{A}=E\otimes \CC$$  is a complex Dirac structure in $E\otimes \CC$.

{Further, the above quadruple forms a   V-data if{f} $J$ is a generalized complex structure.}

\end{lem}
\begin{proof}
Exactly as the proof of  Lemma \ref{keydirac}, but working over $\CC$ and taking $K:=\bar{A}$.
\end{proof}

As in \S \ref{dirac}, let $\cM$ be the (real) degree 2 symplectic manifold with self-commuting degree $3$ function $\Delta$ corresponding to the Courant algebroid $E$.  We have  $C_{\CC}(\cN)=C(\cM)\otimes \CC$.
Since $\Delta$ defines a complex Courant algebroid structure on $E\otimes \CC$ which is the complexification of a (real) Courant algebroid structure on $E$, it follows that
$\Delta\in C(\cM)\subset C_{\CC}(\cN)$. We are interested only in complex Courant algebroid structures on $E\otimes \CC$ which are complexifications of Courant algebroid structures on $E$, so we deform $\Delta$ only within $C(\cM)$.

\begin{cly}\label{cor:gc}
Fix a Courant algebroid $E\to M$ and a generalized complex structure $J$, encoded by a complex Dirac structure $A$.   Let $\cM$ and the V-data 
 $(L,\ga,P,\Delta)$  be as in Lemma \ref{keygcs}. 
Then there exists a  (real) $L_{\infty}[1]$-algebra structure  on $(C(\cM)[2])[1]\oplus \ga$ with the property that
for all $\tilde{\Delta}\in C(\cM)_3$
and small enough
 $\tilde{\Phi}\in \Gamma(\wedge^2A^*)$:
\begin{align*}
 &\begin{cases}
  \Delta+\tilde{\Delta} \text{ defines a Courant algebroid structure on }E\\   
 graph(-\tilde{\Phi}) \text{ is  the }+i\text{-eigenbundle of a generalized complex structure there}
\end{cases}\\
 \Leftrightarrow
& (\tilde{\Delta}[3],\tilde{\Phi}[2]) \text{ is a MC element of } (C(\cM)[2])[1]\oplus \ga.
\end{align*}
\end{cly}
\begin{proof}
Apply Thm. \ref{machine} (which holds over $\CC$ as well) with $\Phi=0$ to obtain the complex $L_{\infty}[1]$-structure
  $({L}[1]\oplus \ga)^P_{\Delta}$. View the latter as a real $L_{\infty}[1]$-structure. Since $\Delta\in C(\cM)[2]$, it follows that  $(C(\cM)[2])[1]\oplus \ga $ is a $L_{\infty}[1]$-subalgebra. {Use   Lemma \ref{keygcs}  
 to phrase the conclusions of Thm.  \ref{machine} in terms of Courant algebroids and generalized complex structures.}
\end{proof}


\begin{rem}
{To see that the above V-data is filtered, proceed exactly as in Remark \ref{rem:convDirac}.}
\end{rem}

\subsection{Deformations of complex structures}\label{sec:cs}

In this subsection we study a special case of the situation considered in \S\ref{sec:gcs}: we study  deformations of  a complex structure  on $M$ to generalized complex structures in the $H$-twisted Courant algebroids $TM\oplus T^*M$, where $H$ ranges through all {(real)} closed $3$-forms.
  Deformations of a complex structure within a fixed  Courant algebroid (the standard one) were studied by Gualtieri in \cite[\S 5.3]{Gu2}.

Fix a complex structure $I$ on a manifold $M$. It gives rise to a generalized complex structure 
$J_I:= \left( \begin{smallmatrix} -I&0\\ 0&I^*\end{smallmatrix} \right)$
 for the standard Courant algebroid $TM\oplus T^*M$, whose $+i$-eigenbundle is the complex Dirac subbundle  $A:=T_{0,1}\oplus T^*_{1,0}$ \cite[\S3]{Gu2}. Here $T_{1,0}$ and $T_{0,1}$ denote the holomorphic and anti-holomorphic
tangent bundle of $(M,I)$ respectively.

{Recall that  
 $\Omega^{k}(M,\RR)= 
 \bigoplus_{p,q\ge 0, p+q=k}\Omega^{{p,q}}(M,\RR)$.} 
 
 {Consider $
  \Gamma(\wedge A^*)=\bigoplus_{r\ge 0,s\ge 0}\Omega^{0,r}(M,\wedge^sT_{1,0})$.}
Fix $p\ge 1$ and 
 $\Theta_i \in \Omega^{0,r_i}(M,\wedge^{s_i}T_{1,0})$ with $s_i\ge 1$ (for $i=1,\dots,p$).
We define a map
$$(\Theta_1^{\sharp}\wedge \dots \wedge \Theta_p^{\sharp})\colon \Omega^{p,q}(M,\RR)\to \Omega^{0,q+\sum_ir_i}(M,\wedge^{-p+\sum_i s_i}T_{1,0})$$
similarly to the map defined at the beginning of  \S \ref{sec:tpois}, but taking into account the differential form part of the $\Theta_i$, which simply gets wedge-multiplied. More precisely, assume $\Theta_i=\omega_i\otimes \pi_i$ with $\omega_i\in\Omega^{0,r_i}(M,\CC),\pi_i\in \Gamma(\wedge^{s_i}T_{1,0})$, let $\alpha\in 
\Omega^{p,0}(M,\CC)$ and $\sigma\in \Omega^{0,q}(M,\CC)$. Then the above map is given by
$$\alpha\wedge \sigma\mapsto \pm \sigma\wedge\omega_1\wedge\dots\wedge\omega_p\otimes \left((\pi_1^{\sharp}\wedge \dots \wedge \pi_p^{\sharp})\alpha\right),$$
where the last expression on the r.h.s. was defined at
 the beginning of  \S \ref{sec:tpois}, and
the sign $\pm$ is the parity of ${pq+\sum_{i=1}^p(p+s_1+\dots+s_{i-1})r_i}$ (using the convention $s_0:=0$).

Further we  define $[\Theta_1,\Theta_2]:= (-1)^{s_1r_2} 
  \omega_1\wedge\omega_2\otimes[\pi_1,\pi_2].$
 
\begin{cly}\label{cor:cplx} 
Let $(M,I)$ be a complex manifold.
There is an $\li[1]$-algebra structure on $$\mathfrak{C}:=
{\Omega(M,\RR)[3]}
\oplus
\bigoplus_{r\ge 0,s\ge 0}\Omega^{0,r}(M,\wedge^sT_{1,0})[2]$$
 whose only non-vanishing multibrackets (up to permutations of the entries) are
\begin{itemize}
\item[a)] the differential, which maps $(H,\Theta)$ to $(-dH,\bar{\partial}\Theta{+H^{0,3}})$,\\
 {where $H^{0,3}$ denotes the component of $H$ lying in $\Omega^{0,3}(M,\RR)$,}
  \item[b)] $\{\Theta_1,\Theta_2\}=(-1)^{r_1+r_2+s_1+1}[\Theta_1,\Theta_2]$ for  $\Theta_i \in \Omega^{0,r_i}(M,\wedge^{s_i}T_{1,0})$,
  \item[c)] $\{H,\Theta_1,\dots,\Theta_p\}\;=\;(-1)^{\sum_{i=1}^p{s_i(p-i)}}(\Theta_1^{\sharp}\wedge \dots \wedge \Theta_p^{\sharp})H$ 
  \\for all $p\ge1$, where   $H\in\Omega^{p,q}(M,\RR)$ and $\Theta_i \in \Omega^{0,r_i}(M,\wedge^{s_i}T_{1,0})$ with $s_i\ge 1$.\end{itemize}
Its Maurer-Cartan elements are exactly pairs $(H[3],\Theta[2])$ where $$H\in {{\Omega^3(M,\RR)}},
\;\;\;\;\;\Theta \in \Omega^{0,2}(M,\CC)\oplus
\Omega^{0,1}(M,T_{1,0})\oplus \Gamma(\wedge^2T_{1,0})$$
satisfy: $dH=0$, and 
$-\Theta$ defines a deformation of $J_I$ to a $-H$-twisted generalized complex structure.
 \end{cly}
\begin{rem} 
1)  The graded vector space $\mathfrak{C}$  is concentrated in degrees $\{-2,\dots,dim_{\RR}(M)-2\}$, and its degree $i$ component is 
$\Omega^{{i+3}}(M,\RR)\oplus
\bigoplus\Omega^{0,r}(M,\wedge^sT_{1,0})$
 for    $r+s=i+2$.

2) We make precise the meaning of  ``$-\Theta$ defines a deformation of $J_I$ to a $-H$-twisted generalized complex structure'': it means that 
$graph(-\Theta)\subset A\oplus \bar{A}=(TM\oplus T^*M)\otimes \CC$
is the $+i$-eigenbundle of a generalized complex structure in the Courant algebroid  $(TM\oplus T^*M,\oo \cdot, \cdot \cc_{-H})$ (the   Courant bracket twisted by $-H$ was defined  in \S\ref{sec:tpois}.)
For instance, if $\Theta=B\in \Omega^{0,2}(M,\CC)$, then 
$graph(-\Theta)=\{X+\xi-\iota_XB:X\in T_{0,1},\xi\in T^*_{1,0}\}$.
\end{rem}

 We make more explicit the Maurer-Cartan condition for the $\li[1]$-algebra of Cor. \ref{cor:cplx}. 
$(H,\Theta)$ is a Maurer-Cartan element if $dH=0$ and the following   equation {of order four} is satisfied:
\begin{equation}\label{mcgc}
\bar{\partial}\Theta{+H^{0,3}}\pm\frac{1}{2}[\Theta,\Theta]
+\Theta^{\sharp}H^{1,2}\pm \frac{1}{2}(\Theta^{\sharp}\wedge\Theta^{\sharp})H^{2,1}{\pm \frac{1}{6}(\Theta^{\sharp}\wedge\Theta^{\sharp}\wedge\Theta^{\sharp})H^{3,0}}=0,
\end{equation}
where the signs $\pm$ depend on $\Theta$.

We spell out three special cases.
When $\Theta=B\in \Omega^{0,2}(M,\CC)$, eq. \eqref{mcgc} is equivalent to $$
\bar{\partial}B{+H^{0,3}}=0 \in \Omega^{0,3}(M,\CC).$$

When $\Theta=\varphi\in \Omega^{0,1}(M,T_{1,0})$, decomposing the l.h.s. of eq. \eqref{mcgc} according to bidegrees, we see that eq. \eqref{mcgc} is equivalent to 
\begin{align*}
\bar{\partial}\varphi+\frac{1}{2}[\varphi,\varphi]&=0 \in \Omega^{0,2}(M,T_{1,0}),\\
{H^{0,3}+}\varphi^{\sharp}H^{1,2}- \frac{1}{2}(\varphi^{\sharp}\wedge\varphi^{\sharp})H^{2,1}{- \frac{1}{6}(\varphi^{\sharp}\wedge\varphi^{\sharp}\wedge\varphi^{\sharp})H^{3,0}}&=0\in \Omega^{0,3}(M,\CC).
\end{align*} The first equation states that $-\varphi$ defines a deformation of $I$ to an (integrable) complex structure $I_{-\varphi}$. The second condition is equivalent to $H$ being of type $(2,1)+(1,2)$ with respect to $I_{-\varphi}$. {This is not surprising, since 
 for any closed $H'\in \Omega^3(M,\RR)$, a complex structure   defines a $H'$-twisted generalized complex structure if{f} 
$H'$ is of  type $(2,1)+(1,2)$ \cite[Ex. 2.14]{Gu2}.}
{(To see the second condition, first verify that  
the evaluation of $H$ on three vectors of the form $X-\varphi(X)$  vanishes, for $X\in T_{0,1}$. Hence the component of $H$   of type $(0,3)$ w.r.t. $I_{-\varphi}$ vanishes. 
 Then use that 
 $H$ is real, to conclude that
$H$ is of type $(2,1)+(1,2)$ with respect to $I_{-\varphi}$.)}

The most interesting case is when $\Theta=\beta\in \Gamma(\wedge^2T_{1,0})$. In that case eq. \eqref{mcgc} is equivalent to 
\begin{align*}
[\beta,\beta] &= 0 \in \Gamma(\wedge^3T_{1,0}),\\
\bar{\partial}\beta+\frac{1}{2}(\beta^{\sharp}\wedge\beta^{\sharp})H^{2,1}&=0 \in \Omega^{0,1}(M,\wedge^2T_{1,0}),\\
\beta^{\sharp}H^{1,2}&=0 \in \Omega^{0,2}(M,T_{1,0}),\\
{ H^{0,3}}&=0 \in \Omega^{0,3}(M,\CC).
\end{align*}
{(Here we used $H^{3,0}=\overline{H^{0,3}}=0$.)}
By the first equation $\beta$ is a Poisson bivector field, however it is not holomorphic in general due to the second equation. {$H$ is of  type $(2,1)+(1,2)$ by the fourth equation.}
 
We do not discuss here the equivalences on the set of Maurer-Cartan elements of $\mathfrak{C}$. We just point out that they are induced by elements of ${\Omega^{2}}(M,\RR)[3]\oplus\Gamma(T_{1,0})[2]\oplus \Omega^1(M,\CC)[2]$.

\begin{proof}[Proof of Cor. \ref{cor:cplx}:]
Apply Cor. \ref{cor:gc} to the standard Courant algebroid $TM\oplus T^*M$ and to the generalized complex structure $J_I$ (i.e., to $A=T_{0,1}\oplus T^*_{1,0}$). It delivers an $\li[1]$-algebra structure on $ (C(\cM)[2])[1]\oplus \ga$ governing deformations of the Courant algebroid and of generalized complex structures.
Recall that, given $H\in \Omega^3(M)$, the degree $3$ function $\Delta+H$ on $\cM$ defines a Courant algebroid structure on  $TM\oplus T^*M$
  if{f} $H$ is closed, and in this case it
induces the $(-H)$-twisted Courant bracket \cite[\S4]{Dima}\cite[\S 8]{HDirac}. 

To conclude the proof, we just need to show that ${\Omega(M,\RR)}[3]\oplus \ga$ is an $\li[1]$-subalgebra of $(C(\cM)[2])[1]\oplus \ga$, and that the restricted multibrackets are  those given in the statement.

 We use the following notation for the canonical local coordinates on $\cN=T^*[2]A[1]=T^*[2](T_{0,1}\oplus T^*_{1,0})[1]$ (i.e. for local generators of $C_{\CC}(\cN)=C(T^*[2]T[1]M)\otimes \CC$).
For $j=1,\dots,dim_{\CC}(M)$ 
  we denote by 
$z_j$ complex local coordinates on $M$, by $\bar{z}_j$ the conjugate coordinates,
 by $p_j$ the canonical coordinates on the fibers of $T^*_{1,0}$
 and by $\bar{v}_j$ those on the fibers of $T_{0,1}$
  (so the degrees are $|z_j|=|\bar{z}_j|=0, |p_j|=|\bar{v}_j|=1$). 
By $P_j,\bar{P}_j,v_j,\bar{p}_j$ we denote the coordinates on the fibres of $T^*[2]A[1]\to A[1]$ conjugate to $z_j,\bar{z}_j, p_j,\bar{v}_j$ respectively (their degrees  are $|P_j|=|\bar{P}_j|=2, |v_j|=|\bar{p}_j|=1$). 

The quadruple listed in Lemma \ref{keygcs} reads:
\begin{itemize}
\item   $L=C_{\CC}(T^*[2]A[1])[2]$, whose Lie bracket we denote by  $\op\cdot,\cdot\cp$
\item   $\ga= C_{\CC}(A[1])[2]\cong \Gamma(\wedge A^*)[2]$ 
\item the natural projection $P \colon L \to \ga$ given by evaluation on the base $A[1]$,
\item $\Delta=\sum_i P_iv_i+\sum_i \bar{P}_i\bar{v}_i$.
\end{itemize}
Notice that $\Delta$ is given essentially by the de Rham differential $d={\partial}+\bar{\partial}$.
The multibrackets of the $\li[1]$-algebra $(L[1]\oplus \ga)^P_\Delta$ are given in Thm. \ref{voronov}.
Clearly ${\Omega(M,\RR)}[2]$ is a Lie subalgebra of $L$ (for it is abelian), and further it is closed under $\op\Delta,\cdot\cp$ since the latter acts as the de Rham differential. By Remark \ref{kerPa} it follows that  $\mathfrak{C}={\Omega(M,\RR)}[3]\oplus \ga$ is a $\li[1]$-subalgebra of 
$(C(\cM)[2])[1]\oplus \ga\subset
(L[1]\oplus \ga)^P_\Delta$. 

The restriction of the  multibrackets  to $\mathfrak{C}$ is the one described in the statement of this corollary, as one computes  in coordinates: a) is obtained from eq. \eqref{diffV2}, b) from eq. \eqref{vorderlong}, and c) from eq. \eqref{vorder}.
\end{proof}

\bibliographystyle{habbrv}
\addcontentsline{toc}{section}{References}
\bibliography{DerbibGeo}
 
\end{document}